\documentclass[a4paper,11pt,reqno]{amsart}

\usepackage{amsmath, multicol, a4wide}
\usepackage[all]{xy}

\newtheorem{theorem}{\bf Theorem}[subsection]
\newtheorem{prop}[theorem]{\bf Proposition}
\newtheorem{cor}[theorem]{\bf Corollary}
\newtheorem{lemma}[theorem]{\bf Lemma}

\newtheorem{definition}[theorem]{\bf Definition}

\theoremstyle{remark}
\newtheorem{example}[theorem]{\bf Example}

\theoremstyle{remark}
\newtheorem{rem}[theorem]{\bf Remark}

 \numberwithin{equation}{subsection}

\newcommand{\Hom}{\operatorname{Hom}}

\def\go{\mathfrak}
\def\bb{\mathbb}
\def\C{\bb C}
\def\Z{\bb Z}
\def\N{\bb N}
\def\R{\bb R}
\def\cal{\mathcal}

 \def\adots{\mathinner{\mkern2mu\raise1pt\hbox{.}
\mkern3mu\raise4pt\hbox{.}\mkern1mu\raise7pt\hbox{.}}}

\date\today

 \title[Invariant differential operators]{Invariant differential operators on a class of multiplicity free spaces}

\author{Hubert Rubenthaler}

\address
{Hubert Rubenthaler\\ Institut de Recherche Math\'ematique Avanc\'ee\\
Universit\'e de Strasbourg et CNRS\\
7 rue Ren\'e Descartes\\
67084 Strasbourg Cedex\\ France\\
E-mail: {\tt rubenth@math.unistra.fr}}

\begin{document}
\parindent=0pt
 
 \maketitle

\begin{abstract}

{\bf If $(G,V)$ is a multiplity free space with a one dimensional quotient we give generators and relations for  the non-commutative algebra $D(V)^{G'}$ of invariant differential operators under the semi-simple part $G'$ of the reductive group $G$. More precisely we show that $D(V)^{G'}$ is the quotient of a Smith algebra   by a completely described two-sided ideal.}

 \end{abstract}
 
\maketitle
\medskip\medskip\medskip\medskip\medskip\medskip
 
 \section{Introduction}
 \vskip 10pt
 Let $H$ be a  reductive algebraic group over ${\C}$ and let $X$ be a smooth irreducible  $H$-variety. Let ${\C}[X]$ be the algebra of regular functions on $X$ and let $D(X)$ be the algebra of differential operators on $X$. Then the $H$-action on $X$ extends naturally to $\C[X]$ and $D(X)$. Let $\C[X]^H$ (resp. $D(X)^H)$ be the subalgebras of $H$-invariants in $\C[X]$ (resp. $D(X)$). The ring $\C[X]^H$ is the ring of regular functions on the categorical quotient $X/\hskip -3pt/H$.  The problem of determining the structure of $D(X)^H$ was investigated by several authors  (\cite{Schwarz},  \cite{Van-den-Bergh}, \cite{Levasseur-Stafford}). On the other hand  under the above mentioned hypothesis there exists  a $H$-equivariant  restriction map
$$ \delta:D(X)^H \longrightarrow D(X/\hskip -3pt /H).$$
obtained by applying  elements in $D(X)^H$ to functions in  $\C[X]^H$. It is  expected that $D(X)^H$ as well as its image under $\delta$ (the so-called algebra of radial components) should share many properties of enveloping algebras(\cite{Schwarz}, \cite{Levasseur}). In this paper we obtain the precise structure of $D(V)^{G'}$ in the  case where $(G,V)$ is a  so called {\it multiplicity free spaces with one dimensional quotient}, (here  $G$ is reductive and  $G'=[G,G]$ is the derived group).  These spaces are defined to be  the multiplicity free spaces $(G,V)$ for which  the quotient $V/\hskip -3pt/G'$ is one dimensional. To be more precise we show that the (non-commutative) algebra $D(V)^{G'}$ is a quotient of a generalized {\it Smith algebra}. Over $\C$ this kind of algebras were introduced by S. P. Smith (\cite{Smith}) as natural generalizations of the enveloping algebra of ${\go {sl}}_{2}$.  As a Corollary we describe by generators and relations the  algebras of radial components attached to the $G'$-isotypic components in the polynomial algebra $\C[V]$ (the image under $\delta$ above corresponds to the trivial representation of $G'$).

According to the classification obtained by the author  (\cite{Rubenthaler-JLT}), the class of multiplicity free spaces with a one dimensional quotient is a rather large class inside the multiplicity free spaces. It contains both irreducible and non irreducible representations. 

The representations $(Str(V), V)$ where $V$ is a simple Jordan algebra over $\C$ and where $StrV)$ is the structure group of $V$ are examples of irreducible multiplicity free spaces with a one dimensional quotient (see Remark \ref{rem-Jordan}  and Example \ref{ex-commutative-PV} below). Among these there is the natural representation of $GL(n,\C)$ on the space $Sym_{n}(\C)$ of $n\times n$ symmetric matrices and also the irreducible 27-dimensional representation of $E_{6}\times \C^*$.

The spin representation of $Spin(7)\times \C^*$ and the irreducible 7-dimensional representation of $G_{2}\times \C^*$ are other irreducible examples.

The representation $(SL(n,\C)\times (\C^*)^2, \Lambda_{1}\oplus \Lambda^2(\Lambda_{1}))$ ($n $ odd and $n\geq 5$) where $\Lambda_{1}$ is the natural representation of $SL(n,\C)$ and where $\Lambda^2(\Lambda_{1}) $ is its second exterior power, provides a non irreducible example.
\vskip 10pt
Let us now give a more precise description of our paper.
 \vskip 3pt

In section 2 we first give the basic definitions, notations and properties of multiplicity free spaces, including multiplicity free spaces with one dimensional quotient.      If $(G,V)$ is a multiplicity free space   then $G$ has an open orbit on $V$  (i.e. $(G,V)$ is a prehomogeneous vector space). We also prove that in the so-called regular case  the $G$-invariant differential operators on the open orbit of a multiplicity free space have always polynomial coefficients (in fact a slightly more general result is proved, see Theorem \ref{th-op-diff-poly}).  
 \vskip 3pt
 In section 3 we introduce the various algebras of differential operators we are interested in.  We define their natural gradings and we define the so-called Bernstein-Sato polynomial of an  homogeneous operator of any degree, not only for degree zero operators as usual. We obtain there the first results concerning these algebras.      
Using the Harish-Chandra isomorphism for multiplicity free spaces (\cite{Knop-2}), we prove a key lemma on invariant polynomials under the so-called little Weyl group which enables us to prove that $D(V)^{G}$ is a polynomial algebra over the center ${\cal Z}({\cal T})$ of $D(V)^{G'}$, with the Euler operator as generator (Theorem \ref{th-key-structure-T0}). We also give generators of  the center ${\cal Z}({\cal T})$ (Theorem \ref{th-Z(T)}) and obtain some specific results in the case of prehomogeneous vector spaces  of commutative parabolic type (Theorem \ref{th-special-commutatif}).   
 \vskip 3pt

 Section 4, which is the main section,  is devoted to the structure of $D(V)^{G'}$. We first briefly define and study the Smith algebras over a commutive ring ${\bf A}$ with unit and no zero divisors (the original definition by Smith was over $\C$). These algebras are defined by generators and relations (involving a polynomial in ${\bf A}[t]$),  and their center is a polynomial algebra ${\bf A}[\Omega_{1}]$, where $\Omega_{1}$ is a generalized Casimir element. Our main result   asserts that $D(V)^{G'}$ is isomorphic to the quotient of a Smith algebra over its center ${\cal Z}({\cal T})$ by the two-sided ideal generated by the element Ê$\Omega_{1}$. Concretely, we give generators and relations for $D(V)^{G'}$ (see Theorem \ref{th-generateurs-relations}).  
 
  \vskip 5pt
  Section 5 is devoted to the study of the algebras of radial components. By radial component of a differential operator in $D(V)^{G'}$ we mean the restriction of $D$ to a $G'$-isotypic component of $\C[V]$. As a corollary of the preceding results we prove that these algebras are quotients of "classical" Smith algebras, that is Smith algebras over $\C$ (see Theorem \ref{th-radial-components}). Of course the defining relations depend on the $G'$-isotypic component. We also give generators of the kernel of the radial component map.  In the case of the trivial representation of $G'$, the structure of the algebra of radial components was first obtained by Levasseur (\cite{Levasseur}), by other methods.


  \vskip 5pt
 \noindent {\bf Acknowledgment:} I would like to thank Thierry Levasseur for providing me with the manuscript of \cite{Levasseur}.   I would also like to thank Sylvain Rubenthaler who provided me with a first proof of Proposition \ref{prop.U-contient-pol} which was important for my understanding.
\vskip 10pt


\section{Multiplicity free spaces with a one dimensional quotient}


\vskip 5pt
\subsection {Prehomogeneous Vector Spaces. Basic definitions and properties} \hfill
 \vskip5pt

  Let $G$ be a  connected algebraic group over ${\bb C}$, and let $(G,\rho, V)$ be a   rational representation of $G$ on the (finite dimensional) vector space $V$. Then the triplet $(G,\rho,V)$ is called a {\it prehomogeneous vector space} (abbreviated to $PV$) if the action of $G$ on $V$ has a Zariski open orbit $\Omega \subset V$.   For the general theory of $PV$'s, we refer the reader to the book  of Kimura \cite{Kimura-book} or to \cite{Sato-Kimura}. The elements in $\Omega$ are called {\it generic}. The $PV$ is said to be {\it irreducible } if  the corresponding representation is irreducible. The {\it singular set} $S$ of  $(G,\rho,V)$ is defined by $S=V\setminus \Omega$. Elements in $S$ are called {\it singular}. If no confusion can arise we often simply denote the $PV$ by $(G,V)$. We will also   write $g.x$ instead of $\rho(g)x$, for $g\in G$ and $x\in V$. It is easy to see that the condition for a rational representation $(G,\rho,V)$ to be a $PV$ is in fact an infinitesimal condition. More precisely let ${\go g}$ be the Lie algebra of $G$ and let $d\rho$ be the derived representation of $\rho$. Then $(G,\rho,V)$ is a PV if and only if there exists $v\in V$ such that the map:
$$\begin{array}{rcl}
{\go g}&\longrightarrow&V\\
X&\longmapsto&d\rho(X)v
\end{array}$$
is surjective (we will often write $X.v$ instead of $d\rho(X)v$). Therefore we will call $({\go g}, V)$ a $PV$ if the preceding condition is satisfied.

Let $(G,V)$ be  a $PV$. A rational function $f$ on $V$ is called a {\it relative invariant} of    $(G,V)$ if there exists a rational character $\chi $ of $G$ such that $f(g.x)=\chi(g)P(x)$ for $g\in G$ and $x\in V$.  From the existence of an open orbit it is easy to see that a character $\chi$ which is trivial on the isotropy subgroup of an element  $x\in \Omega$ determines a unique    relative invariant $P$. Let $S_{1},S_{2},\dots,S_{k}$ denote the irreducible components of codimension one of the singular set $S$. Then there exist irreducible polynomials $P_{1}, P_{2},\dots,P_{k}$ such that $S_{i}=\{x\in V\,|\, P_{i}(x)=0\}$. The polynomials $P_{i}$'s are unique up to nonzero constants, they  are relative invariants of $(G,V)$ and any nonzero relative invariant $f$ can be written in a unique way $f=cP_{1}^{n_{1}}P_{2}^{n_{2}}\dots P_{k}^{n_{k}}$, where $n_{i}\in {\bb Z}$ and $c\in \C^*$ . The polynomials $P_{1}, P_{2},\dots,P_{k}$ are called the {\it fundamental relative invariants} of $(G,V)$. Moreover if the representation $(G,V)$ is irreducible then  there exists at most one irreducible polynomial which is relatively invariant.

 The prehomogeneous vector space  $(G,V)$ is called {\it regular} if there exists a relative invariant polynomial $P$ whose Hessian $H_{P}(x)$ is nonzero on $\Omega$. If $G$ is reductive, then $(G,V)$ is regular if and only if the singular set $S$ is a hypersurface, or if and only if the isotropy subgroup of a generic point  is reductive. If the $PV$ $(G,V)$ is regular, or if $G$ is reductive, then the contragredient representation $(G,V^*)$ is again a $PV$.
 
 \vskip 10pt

\subsection{Multiplicity free spaces}\hfill
\vskip 5pt

 For the results concerning multiplicity free spaces we refer the reader to the survey by Benson and Ratcliff (\cite{Benson-Ratcliff-survey}) or to \cite{Knop-2} .  
Let $(G,V)$ be a finite dimensional rational  representation of a  connected reductive algebraic group $G$. Let $\C[V]$ be the algebra of polynomials on $V$. Then $G$ acts on $\C [V]$ by 
$$g.\varphi(x)=\varphi(g^{-1}x)\qquad(g\in G, \varphi\in \C[V]).$$
As the space $\C[V]^n$ of homogeneous polynomials of degree $n$ is stable under this action, the representation $(G,\C[V])$ is completely reducible.
Let  $D(V)$ be the algebra of differential operators on $V$ with polynomial coefficients. The group $G$ acts also on $D(V)$ by 
$$(g.D)(\varphi)= g.(D(g^{-1}.\varphi))\quad (g\in G, D\in D(V),  \varphi\in \C[V]).$$
Recall the  $G$-equivariant identifications between $\C[V]$ and the symmetric algebra $S(V^*)$ of the dual space $V^*$ and between $\C[V^*]$ and the symmetric algebra $S(V)$ of $V$. The embedding $ \begin{array}{ccc} V& \longrightarrow& D(V)\\
v&\longmapsto &D_{v}
\end{array}$
where $D_{v}P(x)=\lim_{t\rightarrow 0} \frac{P(x+tv)-P(x)}{t}$ extends  uniquely  to an embedding $S(V)\longrightarrow D(V)$ whose image is the ring of differential operators with constant coefficients. If $f\in S(V)\simeq \C[V^*]$ we denote by $f(\partial)$ the corresponding differential operator. Another way to construct $f(\partial)$ for $f\in\C[V^*]$ is to say that $f(\partial)$ is the unique differential operator on $V$ satisfying 
$$f(\partial_{x})e^{\langle x,y \rangle}=f(y)e^{\langle x,y \rangle}\quad (x\in V, y\in V^*)\eqno (2-2-1)$$ 

Recall also that the $\C[V]$-module $D(V)$ can be identified with $\C[V]\otimes S(V)$ through the multiplication map 
$$\begin{array}{ccl}
{}&\simeq&\\
m:\C[V]\otimes S(V)&\longrightarrow&D(V)\\
\varphi\otimes f&\longmapsto& \varphi f(\partial)
\end{array}$$
The preceding map is in fact $G$-equivariant and therefore the $G$-module $D(V)$ is isomorphic to the $G$-module $\C[V]\otimes S(V)$. The duality pairing $V\otimes V^*\longrightarrow {\C}$ extends uniquely to the   non-degenerate $G$-equivariant pairing 
$$\begin{array}{rcl}S(V)\otimes S(V^*)\simeq\C[V^*]\otimes \C[V]&\longrightarrow& \C\\
f\otimes \varphi&\longmapsto&\langle f, \varphi \rangle = f(\partial)\varphi(0) 
\end{array}\eqno(2-2-2)$$
which gives rise to an embedding $\C[V^*]\hookrightarrow \C[V]^*$. It is easy to see that $\langle \C^{i}[V^*],\C^j[V]\rangle =\{0\}$ if $i\neq j$.

\begin{definition} Let $G$ be a  connected reductive algebraic group, and let $V$ be the space of a finite dimensional (complex) rational representation of $G$. The representation $(G,V)$ is said to be multiplicity free (abbreviated to $MF$) if each irreducible representation of $G$ occurs at most once in the representation $(G, \C[V])$.
\end{definition}
\vskip 5pt

 \begin{rem}\label{rem-classification-MF}
 
   Historically the classification of $MF$ spaces goes as follows.  Kac (\cite{Kac}) determined all the $MF$  spaces where the representation $(G,V)$ is irreducible. Brion (\cite{Brion}) did the case where $G'=[G,G]$ is (almost) simple. Finally Benson-Ratcliff and Leahy classified independently, up to geometric equivalence, all the indecomposable saturated MF-spaces (see \cite{Benson-Ratcliff-survey}, \cite{Benson-Ratcliff-article},
  \cite{Leahy},\cite{Knop-2})
 \end{rem}
 
 \medskip

The following theorem summarizes  some  basic results concerning $MF$  spaces (see \cite{Benson-Ratcliff-survey}, \cite{Howe-Umeda}, \cite{Knop-2}):
\begin{theorem}\label{th-caracterisation-MF}\hfill

$1)$  A finite dimensional representation $(G,V)$ is $MF$ if and only if $(B,V)$ is a  prehomogeneous vector space for any Borel subgroup $B$ of $G$ $($and hence each $MF$ space $(G,V)$  is a $PV$$)$.

$2)$  A finite dimensional representation $(G,V)$ is $MF$ if and only if the algebra   $D(V)^G$ of invariant differential operators with polynomial coefficients is commutative.

$3)$ If $(G,V)$ is a MF space, then the dual space $(G,V^*)$  is also MF.
\end{theorem}

\begin{proof}

The first assertion is due to Vinberg and Kimelfeld (\cite{Vinberg-Kimelfeld}), another proof can be found in \cite{Knop-2} .
The second assertion is due to Howe and Umeda (\cite{Howe-Umeda}, Theorem 7.1). For the third assertion note  that as $\langle \C^{i}[V^*],\C^j[V]\rangle =\{0\}$ for $i\neq j$, we obtain that  $f\mapsto \langle f,\,\,\rangle$ is a $G$-equivariant isomorphism between $\C^{i}[V^*]$ and $\C^{i}[V]^*$, and hence $(G,V^*)$ is multiplicity free.

\end{proof}
  
 \vskip 5pt

 Let us be more precise about the decomposition of the polynomials under the action of the group $G$ or a Borel subgroup.  Therefore we need more notations. We can write $G= G'C$ where $G'=[G,G]$ is the subgroup of commutators, and where $C=Z(G)^\circ\simeq {(\C^*)}^p$ is the connected component of the center of $G$. Let $T'$ be a maximal torus in $G'$, and let $B'=T'U$ be a Borel subgroup of $G'$, where $U$ is the nilradical of $B'$. The group $T=T'C$ is a maximal torus in $G$ and  $B=TU$ is a Borel subgroup of $G$. We will denote by ${\go g}, {\go g'}, {\go t}, {\go t'}, {\go c}, {\go b}, {\go b'}, {\go u}$ the corresponding Lie algebras.
Let $R$ be the set of roots of $({\go g}',{\go t}')$, let $\Delta=\{\alpha_{1},\dots,\alpha_{\ell}\}$ be the basis  of simple  roots corresponding to ${\go b}'$ and let $R^+$ be the corresponding set of positive roots.

Denote by  $\Lambda'$  the lattice of weights of $({\go g}',{\go t}')$. We have $\Lambda'=\Z\omega_{1}\oplus \Z\omega_{2}\oplus\dots\oplus \Z\omega_{\ell}$ where the $\omega_{i}'s$ are the fundamental weights. Let $\Lambda'^{+}=\N\omega_{1}\oplus \N\omega_{2}\oplus\dots\oplus \N\omega_{\ell}$ be the  the set of dominant weights. Denote by $X(C)$ the group of algebraic characters of $C$, which we will sometimes  consider as linear forms on ${\go c}$. Set:
$$\Lambda=\Lambda'\oplus X(C),\quad  \Lambda^+=\Lambda'^+\oplus X(C).$$
For $\lambda\in \Lambda^+$ (resp. $\lambda'\in \Lambda'^+$) let us denote by $V_{-\lambda}$ (resp. $V_{-\lambda'}$) an irreducible ${\go g}$-module (resp. ${\go g}'$-module) with highest weight $\lambda $ (resp. $\lambda'$). We use this unusual notation because we want to index the modules occuring in $\C[V]$ by the character of their highest weight polynomial, rather than by the highest weight.
 
 For a multiplicity free space  $(G,V)$ we have the decomposition:
 $$ \C[V]= \bigoplus_{\lambda\in \Lambda^+}V_{-\lambda}^{m(\lambda)}$$
 where $m(\lambda)=0$ or $1$. If $m(\lambda)=1$, then there exists a uniquely defined positive integer $d(\lambda)$ such that $V_{-\lambda}\in \C[V]^{d(\lambda)}$. The integer  $d(\lambda)$ is called the {\it
  degree} of $\lambda$. Let us  denote by $\Delta_{0},\Delta_{1}, \dots,\Delta_{k}, \dots,\Delta_{r}$  the fundamental relative invariants invariants of the $PV $ $(B,V)$, indexed in such a way that $\Delta_{0}, \Delta_{1},\dots,\Delta_{k}$ are the fundamental relative invariants of the $PV$ $(G,V)$ and such that the other invariants are ordered by decreasing degree. We denote by $d_{i}$ the degree of $\Delta_{i}$ ($i=0,\dots,r$). It is worthwhile noticing that at least $\Delta_{r}$ is of degree one as the highest weight vectors of the irreducible components of $V^*$ must occur.    Then any relative invariant of $(B,V)$ is of the form $c\Delta^{\bf a}$ where ${\bf a}=(a_{0},a_{1},\dots,a_{r})\in \Z^{r+1}$ and where $\Delta^{\bf a}=\Delta_{0}^{a_{0}}\dots\Delta_{r}^{a_{r}}$. The non negative integer $r+1$ is called the {\it rank} of the $MF$ space $(G,V)$.  The algebra of $U$-invariants is the subalgebra generated by the $\Delta_{i}$'s, i.e.  $\C[V]^U=\C[\Delta_{0},\dots,\Delta_{r}]$. As the polynomials $\Delta_{i}$ are algebraically independent, this latter algebra is a polynomial algebra. Let  $\lambda_{i}$ be the character of $\Delta_{i}$ (we use the same notation $\lambda_{i}$ for the character of the group and for its derivative, which is  an element of $\Lambda^+$). Hence the (infinitesimal) character of $\Delta^{\bf a}$ is $\lambda_{\bf a}= a_{0}\lambda_{0}+\dots+a_{r}\lambda_{r}$. Of course by definition the elements $\Delta^{\bf a}$    ($a_{i}\geq 0, i=0,\dots,r$) are the highest weights vectors in $\C[V]$. Due to the fact that the group action on $\Delta^{\bf a}$ is given by $g.\Delta^{\bf a}(x)= \Delta^{\bf a}(g^{-1}x)$, the infinitesimal highest weight of $\Delta^{\bf a}$ is $-\lambda_{\bf a}= -a_{0}\lambda_{0}-\dots-a_{r}\lambda_{r}$.  
  
  If we set $V_{\bf a}=V_{-\lambda_{\bf a}}$, we therefore can write 
$$\C[V]=\bigoplus _{a_{0}\geq0,\dots,a_{r}\geq0}V_{\bf a} \eqno (2-2-3)$$
Sometimes, if $\lambda= a_{0}\lambda_{0}+\dots+a_{r}\lambda_{r}$, we simply write $V_{\lambda}$  instead of $V_{\bf a}$.
If we denote by $d_{i}$ the degree of $\Delta_{i}$, one can notice that all elements in $V_{\bf a}$ are of  degree  $d({\bf a})=a_{0}d_{0}+a_{1}d_{1}+\dots+a_{r}d_{r}$. It is also worthwhile noticing that  we have:
$$V_{\bf a}=\Delta_{0}^{a_{0}}\Delta_{1}^{a_{1}}\dots\Delta_{k}^{a_{k}}V_{0,\dots,0,a_{k+1},\dots,a_{r}}. \eqno(2-2-4)$$

\medskip

The proof of the following lemma is straightforward.

\begin{lemma}\label{lemma-decomp-fractions}\hfill

 Define ${\cal O}=\{x\in V\,|\, \Delta_{i}(x)\neq 0, i=0,\dots,k\}$. Let $\C[{\cal O}]$ be the ring of regular functions on ${\cal O}$ $($elements of $\C[{\cal O}]$ are just rational functions whose denominators are of the form $\Delta_{0}^{a_{0}}\dots\Delta_{k}^{a_{k}}$, with $a_{0},\dots,a_{k}\geq 0$$) $. As the polynomials $\Delta_{0},\dots,\Delta_{k}$ are relative invariants under $G$, the open set ${\cal O}$ is $G$-stable, and therefore $G$ acts on $\C[{\cal O}]$. Then $\C[{\cal O}]$ decomposes without multiplicities under the action of $G$. More precisely the decomposition into irreducibles is given by
$$\C[{\cal O}]=\bigoplus_{\begin{array}{c}
(a_{0},\dots,a_{k})\in \Z^{k+1}\\
(a_{k+1},\dots,a_{r})\in \N^{r-k}
\end{array}}V_{\bf a}$$
where $V_{\bf a}=\Delta_{0}^{a_{0}}\Delta_{1}^{a_{1}}\dots\Delta_{k}^{a_{k}}V_{0,\dots,0,a_{k+1},\dots,a_{r}}$ is the irreducible subspace of $\C[{\cal O}]$ generated by the highest weight vector $\Delta^{\bf a}=\Delta_{0}^{a_{0}}\Delta_{1}^{a_{1}}\dots\Delta_{r}^{a_{r}}$.  

\end{lemma}



\begin{rem} We want to draw the attention of the reader to the fact that if $(G,V)$ is not a regular $PV$, then the open set ${\cal O} $ may be distinct from the open $G$-orbit $\Omega$.

\end{rem}
The preceding  lemma has the following consequence. 
\begin{theorem} \label{th-op-diff-poly}\hfill 

Let $(G,V)$ be a multiplicity free space.  As before set  ${\cal O}=\{x\in V\,|\, \Delta_{i}(x)\neq 0, i=0,\dots,k\}$. Then $D(V)^G=D({\cal O})^G$. In other words any $G$-invariant differential operator with coefficients in $\C[{\cal O}]$ has in fact polynomial coefficients.

\end{theorem}

\begin{proof}

Let $D\in D({\cal O})^G$. As we know from the preceding Lemma that $\C[{\cal O}]$ decomposes without multiplicities under $G$, we obtain that $D$ defines a $G$-equivariant endomorphism on each $V_{\bf a}$, ${\bf a}\in \Z^{k+1}\times \N^{r-k}$. Therefore $D$ stabilizes $\C[V]=\bigoplus _{a_{0}\geq0,\dots,a_{r}\geq0}V_{\bf a}$. It is easy to see that a differential operator with rational coefficients and which stabilizes the polynomials must have polynomial coefficients.

\end{proof}

\begin{rem}\label{rem-Jordan}   Let $V$ be a simple Jordan algebra over $\C$ or $\R$. Let $\Omega$ be the set of invertible elements in $V$ and let  $G$ be the structure group of $V$. It is known  that $(G,V)$ is a multiplicity free space with $\Omega$ as open $G$-orbit. Then the preceding  theorem  implies that $D(V)^G=D(\Omega)^G$.  This result was already known in this context and is  usually obtained by computing  an explicit set of generators of $D(\Omega)^G$ (see \cite{Nomura}, \cite{Yan} or  \cite{Faraut-Koranyi-book}).  Through the so-called Kantor-Koecher-Tits construction there is a one-to-one correspondence between these spaces and  the $PV$'s of commutative parabolic type (see  example \ref{ex-commutative-PV} below).

\end{rem}

\begin{prop}\label{prop-G'-isotypic} \hfill

Let $(G,V)$ be a $MF$ space.
For $\widetilde{\bf a}=(a_{k+1},\dots,a_{r})\in \N^{r-k}$ we define $V_{\widetilde{\bf a}}=V_{(0,\dots,0,a_{k+1},\dots,a_{r})}$. Then for ${\bf {a}}=(a_{0},\dots,a_{k},a_{k+1},\dots,a_{r})$ the spaces $V_{\bf a}=\Delta_{0}^{a_{0}}\dots \Delta_{k}^{a_{k}}V_{\bf {\tilde a}}$ are $G'$-equivalent if ${\bf {\tilde a}}$ is fixed and if $(a_{0},\dots,a_{k})\in \Z^ {k+1}$. If we define
$$U_{{\bf {\tilde a}}}=\bigoplus_{(a_{0},\dots,a_{k})\in \N^{k+1}}\Delta_{0}^{a_{0}}\dots \Delta_{k}^{a_{k}}V_{\bf {\tilde a}},\quad  W_{{\bf {\tilde a}}}=\bigoplus_{(a_{0},\dots,a_{k})\in \Z^{k+1}}\Delta_{0}^{a_{0}}\dots \Delta_{k}^{a_{k}}V_{\bf {\tilde a}},$$
then the decompositions of $\C[V]$ and $\C[\cal O]$ into $G'$-isotypic components are given by 
$$\C[V]= \bigoplus_{{\bf \tilde a}}U_{{\bf \tilde a}}, \quad \C[{\cal O}]= \bigoplus_{{\bf \tilde a}}W_{{\bf \tilde a}}$$
\end{prop}

\begin{proof} The map $P\longmapsto \Delta_{0}^{a_{0}}\dots \Delta_{k}^{a_{k}}P$ is a $G'$-equivariant isomorphism between $ V_{\bf {\tilde a}}$
 and 
  $\Delta_{0}^{a_{0}}\dots \Delta_{k}^{a_{k}}V_{\bf {\tilde a}}$, hence all these spaces are $G'$-equivalent. To prove the second assertion it is enough to prove that if ${\bf \tilde a}\neq {\bf \tilde b}$, then the spaces $V_{ {\bf {\tilde a}}}$ and $V_{ {\bf {\tilde b}}}$ are not $G'$-equivalent.  Suppose that this would be the case and let $\Delta^{\bf \tilde a}$ and $\Delta^{\bf \tilde b}$ be the corresponding highest weight vectors with characters $\lambda_{\bf \tilde a}$ and $\lambda_{\bf \tilde b}$ respectively. From the $G'$-equivalence we know that ${\lambda_{\bf \tilde a}}_{|_{{{\go t}'}}}={\lambda_{\bf \tilde b}}_{|_{{{\go t}'}}}$ and hence $P=\frac{\Delta^{\bf \tilde a}}{\Delta^{\bf \tilde b}}$ is a relative invariant under $B$ whose character is trivial on ${\go t}'$. Therefore it generates a one dimensional representation, hence $P$ is a relative invariant under $G$. Finally we obtain that $\Delta^{\bf \tilde a}= \Delta_{0}^{a_{0}}\dots\Delta_{k}^{a_{k}}\Delta^{\bf \tilde b}$, and this is not possible if ${\bf \tilde a}\neq {\bf \tilde b}$.
 
 \end{proof}
 
 As $(G,V^*)$ is  multiplicity free (Theorem \ref{th-caracterisation-MF}) and as   $\C^{i}[V^*]\simeq \C^{i}[V]^*$, we have 
 $$\C[V^*]= \bigoplus_{a_{0}\geq 0,\dots,a_{r}\geq 0}V_{\bf a}^*\eqno(2-2-5)$$
 where $V_{\bf a}^*$ is the irreducible $G$-submodule of $\C[V^*]$ generated by a lowest weight vector ${\Delta^*}^{\bf a}\in \C[V^*]$, defined up to a multiplicative constant,  whose character with respect to the opposite Borel subgroup $B^-$ is equal to $-\lambda_{\bf a}=-a_{0}\lambda_{0}-\dots-a_{r}\lambda_{r}$. Let us  fix   a lowest weight vector $\Delta_{i}^*$ ($i=0,\dots,r$) with character $-\lambda_{i}$ (with respect to $B^{-}$). Then we can choose ${\Delta^*}^{\bf a}={\Delta_{0}^*}^{a_{0}}{\Delta_{1}^*}^{a_{1}}\dots{\Delta_{r}^*}^{a_{r}}$.  Of course the  module $V_{\bf a}^*$ is the dual module of $V_{\bf a}$ through $f\longmapsto \langle f,\,\, \rangle$ (see $(2-2-2))$.
 \vskip 5pt
 
 As $V_{\bf a}$ is a $G$-irreducible module, it is well known that the tensor $G$-module $V_{\bf
a}\otimes V_{\bf a}^*$ contains up to constant, a unique $G$-invariant vector $R_{\bf a}$ and that  $V_{\bf a}\otimes V_{\bf b}^*$ does not contain any non trivial $G$-invariant vector if ${\bf a}\neq {\bf b}$ (see for example
\cite{Howe-Umeda}). To be more precise we define $R_{\bf a}$ to be the operator corresponding to the "unit matrix" in  $V_{\bf
a}\otimes V_{\bf a}^*\simeq \Hom(V_{\bf a},V_{\bf a})$. Moreover as
$  { \bb C}[V]\otimes{ \bb C} [V^*]$ is $G$-isomorphic to ${D}(V)$, the element $R_{\bf a}$
can be viewed as a $G$-invariant differential operator with polynomial coefficients. The operators $R_{\bf a}$
are sometimes called {\it Capelli operators}. They are also called {\it unnormalized canonical invariants} in \cite{Benson-Ratcliff-survey}. Moreover the family of elements  $R_{\bf a}$
 (${\bf a}\in { \bb N} ^{r+1}$)  is a vector basis of the vector space ${ D}(V)^G= {
D}({\cal O})^G$.

The Capelli operators $R_{i}$ corresponding to the space $V_{\lambda_{i}}$ $(i=0,\dots,r)$ will be of particular importance because of the result below.  

\begin{theorem}\label{th-generateurs-op-inv} {\rm(Howe-Umeda)}\hfill

Let $(G,V)$ be a $MF$ space. The Capelli operators $R_{i}$ $(i=0,\dots,r)$ are algebraically independent and $D(V)^G=\C[R_{0},\dots, R_{r}].$

\end{theorem} 

\begin{proof} See \cite{Howe-Umeda} (Theorem 9.1) or \cite{Benson-Ratcliff-survey} (Corollary 7.4.4).

\end{proof}

\begin{rem}\label{rem-Capelli}\hfill

a) Recall that for $i=0,1,\dots,k$ the polynomials $\Delta_{0},\Delta_{1},\dots,\Delta_{k}$ are the fundamental relative invariants under the action of the full group $G$. Once these polynomials are fixed, let us define the polynomial $\Delta_{i}^*\in \C[V^*]$ as the unique fundamental relative invariant of $(G,V^*)$   with character $\lambda_{i}^{-1}$, such that $\Delta_{i}^*(\partial)\Delta_{i}(0)=1$, for $i=0,\dots,k$. Then the Capelli operators $R_{i}$ $(i=0,\dots,k)$ are given by $R_{i}=\Delta_{i}(x)\Delta_{i}^*(\partial) $, and the Capelli operator corresponding to the irreducible component $V_{ a_{0}\lambda_{0}+\dots+a_{k}\lambda_{k}}$ is scalar multiple of $\Delta_{0}^{a_{0}}(x)\dots\Delta_{k}^{a_{k}}(x){\Delta_{0}^*} (\partial)^{a_{0}}\dots{\Delta_{k}^*(\partial)}^{a_{k}}$. More generally the Capelli operator $R_{\bf a}$ corresponding to $V_{\bf a}$ where ${\bf a}=a_{0}\lambda_{0}+\dots+a_{k}\lambda_{k}+a_{k+1}\lambda_{k+1}+\dots+a_{r}\lambda_{r}$ is a scalar multiple of  
$ \Delta_{0}^{a_{0}}(x)\dots\Delta_{k}^{a_{k}}(x) R_{a_{k+1}\lambda_{k+1}+\dots+a_{r}\lambda_{r}}{\Delta_{0}^*} (\partial)^{a_{0}}\dots{\Delta_{k}^*(\partial)}^{a_{k}}$. 

b) Moreover, in the case where $(G,V)$ is irreducible, as  $\Delta_{r} $ is  the highest weight vector in $V^*$,  the operator $R_{r}$ is nothing but the Euler operator $E$.

c) More generally, if $V=V_{1}\oplus \dots \oplus V_{\ell}$ where the representations $(G,V_{i})$ are irreducible,  the various Euler operators $E_{i}$ on  $V_{i}$ are  the Capelli operators associated to the irreducible subspaces $V_{i}^*\in \C[V]$. Of course the global Euler operator $E$ on $V$ is given by $E=E_{1}+\dots+E_{\ell}$. As the highest weight vectors of the spaces $(G,V_{i}^*)$ occur as the $\ell$ last elements of  $\Delta_{0},\dots,\Delta_{r}$, we have  $R_{r-\ell+1}=E_{1},\dots, R_{r}=E_{\ell}.$

d) According to b) and c) before, one can always take $\{R_{0},R_{1},\dots, R_{r-1}, E\}$ as a set of algebraically independent generators of $D(V)^G$.

\end{rem}
\vskip 15pt
   \subsection{Multiplicity free spaces with one dimensional quotient}\hfill 
   \vskip 5pt
   
   Let us now define the main objects this paper deals with, namely the $MF$-spaces with a one dimensional quotient, which were introduced by T. Levasseur.
  \begin{definition}
   {\rm (T. Levasseur, \cite{Levasseur} sections 3.2 and  4.2)}\hfill
   
   $1)$ A prehomogeneous vector space $(G,V)$  is said to be of rank one\renewcommand{\thefootnote}{\fnsymbol{footnote}}\footnote{It must be remarked that if $(G,V)$ is also multiplicity free, then the rank as a $PV$ is not at all the same as the rank as a $MF$ space.} if there exists an homogeneous polynomial $\Delta_{0}$ on $V$ such that $\Delta_{0}\notin \C[V]^G$ and such that $\C[V]^{G'}=\C[\Delta_{0}]$.

  $2)$ A multiplicity free space $(G,V)$   is said to have a one-dimensional quotient if it is a $PV$ of rank one.
 \end{definition}

 \begin{rem}\label{rem-pv-rang-1}\hfill
 
  a) The classification of multiplicity free spaces with a one dimensional quotient has been obtained by the author (\cite{Rubenthaler-JLT}).

b)  It can be shown that if $(G,V)$ is a $PV$ of rank one, then
  the polynomial $\Delta_{0}$ is the unique fundamental relative invariant of $(G,V)$. More precisely a $PV$ $(G,V)$ is of rank one if and only if it has a unique fundamental relative invariant (see \cite{Rubenthaler-JLT}). Hence in the notations of section $2.2$ we have $k=0$, in other words $\Delta_{0} $ is the unique fundamental $G$ relative invariant among the $B$ relative invariants $\Delta_{0}, \Delta_{1}, \dots, \Delta_{r}$.
  
   \end{rem}
  
  \medskip 
  We give now some examples of $MF$-spaces with a one dimensional quotient.
  \begin{example}\label{ex-commutative-PV}:  $PV$'s of commutative parabolic type (for details we refer to \cite{Muller-Rubenthaler-Schiffmann}, \cite{Rubenthaler-Schiffmann-1} is also relevant).
  
  Let $\widetilde{\go g}$ be a simple complex Lie algebra. Assume we are given a $3$-grading of $\widetilde{\go g}$:
  $$\widetilde{\go g}=V^-\oplus \go{g} \oplus V^+.$$
Then $\go{g}$ is a reductive Lie subalgebra and it is well known that the representation $(\go{g}, V^+)$ is prehomogeneous (here $\go{g}$ acts on $V^+$ via the bracket).  Let $\widetilde G$ be the adjoint group of $\widetilde{\go g}$ and let $G$ be the connected subgroup of $\widetilde G$ whose Lie algebra is $\go g$. Then the space $(G,V^+)$ is multiplicity free.   Moreover such a space has  a one dimensional quotient if an only if it is a regular $PV$. Up to local isomorphism one obtains the following list:

1) $(SL(n,\C)\times SL(n,\C)\times \C^*, M_{n}(\C))$ acting by $(g_{1},g_{2},t).x=t g_{1}x
g_{2}^{-1} $ $g_{1},g_{2}\in SL(n,\C), t\in \C^*, x\in M_{n}(\C))$; here $\Delta_{0}(x)=\det(x)$.

2) $(O(n,\C)\times \C^*, \C^n)$ with the natural action. Here $\Delta_{0}(x)=Q(x)=\sum_{i=1}^{i=n}x_{i}^2$.

3) $(GL(n,\C), \text{Sym}_{n}(\C))$, where $\text{Sym}_{n}(\C)$ denotes the $n\times n$ symmetric matrices, with the action $g.x=gx{^t\kern -1pt g}$. Then $\Delta_{0}(x)=\det (x)$.

4) $(GL(n,\C), \text{Skew}_{n}(\C))$, $n$ even,  where $\text{Skew}_{n}(\C)$ denotes the $n\times n$  skew-symmetric matrices, with the action $g.x=gx{^t\kern -1pt g}$. Then $\Delta_{0}(x)= {Pf(x)}$, where $Pf(x)$ denotes the pfaffian of the even skew-symmetric matrix $x$.

5) $(E_{6}\times \C^*, \C^{27})$ (the irreducible $27$-dimensional representation of $E_{6}$). The fundamental relative invariant is of degree 3, it is known as the Freudenthal cubic.

\end{example}

\begin{example} $(GL(2)\times Sp(n), \C^2\otimes \C^{2n})$ (tensor product of the natural representations). Here the action is given by 
$$(g_{1}, g_{2}).X= g_{2}X({^t}\kern -3pt g_{1}), \, g_{1}\in SL(2), g_{2} \in Sp(n), X\in M_{2n,2}$$
The relative invariant $\Delta_{0}$ is given by $Pf(^tX JX)$ where $J=\begin{pmatrix}0&Id_{n}\\
-Id_{n}&0
\end{pmatrix}$, and where $Pf(.)$ is the pfaffian of a $2\times 2$  skew symmetric matrix. The rank is equal to $3$ and it is a regular PV.
For details see \cite {Howe-Umeda} (case 11.6)  and \cite{Rubenthaler-JLT} (case 4.1.7).\end{example}  

\begin{example} $(GL(n)\times GL(n-1), M_{n,1}\oplus M_{n,n-1})$. The action is given by 
$$(g_{1},g_{2})(v,x)=(g_{1}v,g_{1}xg_{2}^{-1}), g_{1}\in GL(n),g_{2}\in GL(n-1),v\in M_{n,1},x\in M_{n,n-1}.$$
The relative invariant $\Delta_{0}$ is given by $ \Delta_{0}(x)= det(v;x)$   where $(v;x)$ is the $n\times n$ matrix obtained by putting the column vector $v$ left to the $n\times n-1$ matrix $x$. The rank is equal to $2n-1$ and it is a regular PV.
For details see     \cite{Rubenthaler-JLT} (case 4.2.5) and  \cite{Benson-Ratcliff-australian} (case 4.2.4).

\end{example}

 







   \vskip 5pt

\section{Algebras of differential operators}

From now on we suppose that $(G,V)$ is a $MF$ space with a one dimensional quotient.
\vskip 5pt
 \subsection{Gradings and Bernstein-Sato polynomials}   \hfill
  \vskip 5pt

  Recall that we denote by $\Delta_{0}, \Delta_{1},\dots,\Delta_{r}$ the fundamental relative invariants under a fixed Borel subgroup $B$ of $G$. As the space has a one dimensional quotient, $\Delta_{0}$ is the unique polynomial among them which is relatively invariant under $G$ (this means that $k=0$ in the  notations of section 2.2). We also set ${\cal O}=\{ x\in V\,|\, \Delta_{0}(x)\neq 0\}$.

 Of course the   Euler operator $E$ on $V$  is defined for
 $P\in \C[V]$ by
  $EP(x)= \frac{\partial}{\partial t}P(tx)_{t=1}=P'(x)x $
  is invariant by any element in $GL(E)$.
  
  Once and for all we also define the following two elements in $D(V)$:
  $$X=\Delta_{0}\,\, \text{ (multiplication by } \Delta_{0}), \quad Y= \Delta_{0}^*(\partial).$$
  The operator 
  $$X^{-1} \,\,  \text{ (multiplication by } \Delta_{0}^{-1})$$
  which belongs to $D({\cal O})$ will also play an important role.
  From the definition of the $G$ action on $\C[V]$ and on $D(V)$ we have 
  $$g.X=\lambda_{0}(g^{-1})X,\quad  g.X^{-1}=\lambda_{0}(g)X^{-1}, \quad g.Y=\lambda_{0}(g)Y \eqno(3-1-1)$$
  and hence $X,Y\in D(V)^{G'}$ and $X^{-1}\in D({\cal O})^{G'}.$
   
    Let us now introduce  the following notations that we will use in the rest of the paper:
    $${\cal T}= D({\cal O})^{G'},\quad {\cal T}_{0}=D(V)^G=D({\cal O})^G $$
    
     (the last equality comes from Theorem \ref{th-op-diff-poly}).   Remember that ${\cal T}_{0}$ is a polynomial algebra in $r+1$ variables (Theorem \ref{th-generateurs-op-inv}). We have the following inclusions:
  $${\cal T}_{0}=D(V)^G=D({\cal O})^G\subset D(V)^{G'} \subset {\cal T}=D({\cal O})^{G'}.$$
 An element $D$ in ${\cal T}$ is said to be of degree $m$ if $[E,D]=mD$. As differential operators in ${\cal T}$ have coefficients which are fractions whose denominators are homogeneous (powers of $\Delta_{0}$), it is clear that ${\cal T}$ is graded by its homogeneous components . But on the other hand any homogeneous element $D$ in ${\cal T}$ preserves the $G'$-isotypic components $W_{\bf \tilde a}=\oplus _{n\in \N}\Delta_{0}^nV_{\bf \tilde a}$ (see Proposition \ref{prop-G'-isotypic}). Therefore an homogeneous element $D$   maps $\Delta_{0}^nV_{\bf \tilde a}$ on $\Delta_{0}^{n+j}V_{\bf \tilde a}$  for some $j$ and hence only multiples  of $d_{0}$ (the degree of $\Delta_{0}$) occur as homogeneous degrees in ${\cal T}$. 
If we define, for $p\in \Z$, 
${\cal T}_{p}=\{D\in {\cal T}\,|\, [E,D]=pd_{0}D\}$, then
 $${\cal T}= \oplus_{p\in \Z}{\cal T}_{p} \eqno(3-1-2)$$
 (At this point it is not completely evident that the two definitions of ${\cal T}_{0}$ coincide, that is that $D(V)^G=\{D\in {\cal T}\,|\,[E,D]=0\}$. This will be a consequence of the proof of Proposition \ref{prop-poly-XY-graduation} below).

Similarly if we define 
$$D(V)^{G'}_{p}=\{D\in D(V)^{G'}\,|\, [E,D]=pd_{0}D\},$$
we have $D(V)^{G'}= \bigoplus_{p\in \Z}D(V)^{G'}_{p}$.

\begin{definition}\label{def-poly-Bernstein-Sato} For ${\bf a}=(a_{0},a_{1},\dots,a_{r})$ and $p\in \N$, we define ${\bf a}+p=(a_{0}+p,a_{1},\dots,a_{r})$. Then if $D\in {\cal T}_{p}$,  the Schur Lemma ensures that if $P\in V_{\bf a}$ we have $DP=b_{D}({\bf a})X^pP$ where $b_{D}({\bf a})\in \C$. It is easy to see that $b_{D}$ is a polynomial in the variables $(a_{0},a_{1},\dots,a_{r})$ $($see for example \cite{Knop-2} , proof of Corollary 4.4$)$. This polynomial is called the Bernstein-Sato polynomial of $D$.
\end{definition}

\begin{example}\label{example-b_Y-commutative} Relations $(3-1-1)$ imply that $X\in {\cal T}_{1}$, $X^{-1}\in {\cal T}_{-1}$ and  $Y\in {\cal T}_{-1}$. And of course $E\in {\cal T}_{0}$. Obviously, from the definition, we have $b_{X}( {\bf a})=b_{X^{-1}}({\bf a})=1$, $b_{E}({\bf a})=d_{0}a_{0}+d_{1}a_{1}+\dots+d_{r}a_{r} =$ the degree of $V_{\bf a}$ (recall that $d_{i}$ is the degree of $\Delta_{i}$). From $(3-1-1)$ we obtain that $Y\in {\cal T}_{-1}$. The computation of $b_{Y}$ is  more difficult. However it is known in the case of $PV$'s of commutative parabolic type (see example \ref{ex-commutative-PV}). In this case, for ${\bf X}=(X_{0},X_{1},\dots,X_{r})$ it is given by
$$b_{Y}({\bf X})=c \prod_{j=0}^r(X_{0}+\dots+X_{j}+j\frac{d}{2})\eqno (3-1-3)$$
where the constant $c$ can be made explicit (see \cite{Bopp-Rubenthaler-complexe}, Th\'eor\`eme 3.19) and where $\frac{d}{2}=\frac{\dim(V)-d_{0}}{(d_{0}-1)d_{0}}$. This explicit computation of the polynomial $b_{Y} $ in the particular case of $PV$'s of commutative parabolic type has been obtained by several authors, using dictinct methods (see \cite{Bopp-Rubenthaler-complexe}, \cite{Faraut-Koranyi-book},\cite{Kostant-Sahi}, \cite{Wallach}). The constant $d$ is the same as the constant $d$ which is familiar to specialists of Jordan algebras.  
\end{example}


The following Lemma is  obvious, but useful.

\begin{lemma}\label{lemma-Bernstein-Sato}\hfill

Let $D_{1},D_{2}\in {\cal T}_{p}$. Then $D_{1}=D_{2}$ if and only if $b_{D_{1}}=b_{D_{2}}$.
\end{lemma}

  \begin{definition}\label{def-tau} The automorphism $\tau$ of ${\cal T}=D({\cal O})^{G'} $ is defined by $$\forall D \in {\cal T}, \, \tau(D)=XDX^{-1}$$
  \end{definition}
  \begin{prop}\label{prop-stabilite-tau}\hfill
  
  The algebra ${\cal T}_{0}$ is stable under $\tau$ and for any $D\in {\cal T}_{0}$ we have
  
  $$XD=\tau(D)X\eqno(3-1-4)$$
  $$DY=Y\tau(D)\eqno(3-1-5)$$
  \end{prop}
  \begin{proof} By definition ${\cal T}_{0}=D(V)^G$. From relations $(3-1-1)$ we see that if $D$ is $G$-invariant so is $\tau(D)$. Obviously $\tau(D)\in D( {\cal O})^G$. But $D({\cal O})^G=D(V)^G$ by Theorem \ref{th-op-diff-poly}, hence  ${\cal T}_{0}$ is $\tau$-stable. Relation $(3-1-4)$ is just the definition of $\tau$. We will now prove that $(3-1-5)$ holds on each subspace $V_{\bf a}$. Let $b_{D}$ be the Bernstein-Sato polynomial of $D$. Then an easy calculation shows that the left and right side of $(3-1-5)$ act on $V_{\bf a}$ by $b_{D}({\bf a}-1)b_{Y}({\bf a})X^{-1}$. Then Lemma	\ref{lemma-Bernstein-Sato} implies $(3-1-5)$.
  
  \end{proof} 
  
  Let us denote by ${\cal T}_{0}[X,Y]$ the subalgebra of ${\cal T}$ generated by ${\cal T}_{0}$, $X$ and $Y$. From the preceding Proposition and from the fact that $XY$ and $YX$ belong to ${\cal T}_{0}$ we know that any element $D\in {\cal T}_{0}[X,Y]$ can be written as a finite sum $D=\sum_{p,q\in \N}a_{p,q}X^pY^q$ with $a_{p,q}\in {\cal T}_{0}$. Similarly, let ${\cal T}_{0}[X,X^{-1}]$ denote the subalgebra of ${\cal T}$ generated by ${\cal T}_{0}$, $X$ and $X^{-1}$. Also any element $D$ in ${\cal T}_{0}[X,X^{-1}]$ can we written as a finite sum $D= \sum_{p\in \Z} a_{p}X^p$. The following Proposition shows that  $D(V)^{G'}={\cal T}_{0}[X,Y]$ and that   ${\cal T}=D({\cal O})^{G'}={\cal T}_{0}[X,X^{-1}]$ and makes the gradings   more precise.

\begin{prop}\label{prop-poly-XY-graduation}\hfill

$1)$  We have 

\centerline{$D(V)^{G'}={\cal T}_{0}[X,Y]= (\bigoplus_{p\in \N^*}{\cal T}_{0}Y^p\bigoplus)\bigoplus{\cal T}_{0}\bigoplus (\bigoplus_{p\in \N^*}{\cal T}_{0}X^p)$}

$($in particular $D(V)^{G'}_{p}={\cal T}_{0}X^p $ if $p\geq 0$, and $D(V)^{G'}_{p}={\cal T}_{0}Y^{-p} $ if $p< 0$$)$.
Equivalently we have 

\centerline{$D(V)^{G'}={\cal T}_{0}[X,Y]= (\bigoplus_{p\in \N^*}Y^p{\cal T}_{0} \bigoplus)\bigoplus{\cal T}_{0}\bigoplus (\bigoplus_{p\in \N}X^p{\cal T}_{0})$.}

$2)$
 We have ${\cal T}=D({\cal O})^{G'}={\cal T}_{0}[X,X^{-1}]=\bigoplus_{p\in \Z}{\cal T}_{0}X^p=\bigoplus_{p\in \Z}X^p{\cal T}_{0} $.  
 
 $3)$
  Any element $D$ in ${\cal   T}_0[X,Y] $ can be written uniquely in the form
$$D=\sum_{i> 0} u_iY^i    +  \sum_{i\geq 0} v_iX^i\text {  or  } D=\sum_{i> 0}  Y^i u_{i}   +  \sum_{i\geq 0}  X^iv_i\text {  {\rm (finite  sums)}} $$
with $u_i, v_{i}\in  {\cal   T}_0$.

 Any element  $D\in  {\cal   T}$ can be written uniquely in the form 
$$ 
D=\sum_{i\in {\bb Z}}u_i X^i   \text{  or  }  
 D=\sum_{i\in {\bb Z}}  X^i u_i \quad{\rm (finite \,\,sums)}  
$$
with $u_i\in  {\cal   T}_0$. 
  
\end{prop}

\begin{proof}
$1)$ For the moment we define ${\cal T}_{0}$ by ${\cal T}_{0}=D(V)^G$. From Proposition \ref{prop-G'-isotypic} we know that the decomposition of $\C[V]$ into $G'$-isotypic components is given by 
$$\C[V]= \bigoplus_{{\bf \tilde a}\in \N^r}U_{\bf \tilde a} \text { where }U_{\bf \tilde a}= \bigoplus_{a_{0}\in \N}\Delta_{0}^{a_{0}}V_{\bf \tilde a} \text{ and  } {\bf \tilde a}=(0,a_{1},\dots,a_{r}).$$
We will now use the technique of Howe and Umeda (\cite{Howe-Umeda})  which we have already mentioned before Theorem \ref{th-generateurs-op-inv}. As $\C[V]\otimes \C[V^*]$ is $G'$-isomorphic to $D(V)$, each subspace $\Delta^{a_{0}}V_{\bf \tilde a}\otimes (\Delta^{b_{0}}V_{\bf \tilde a})^*$ will give rise to a unique $G'$-invariant differential operator $R_{a_{0},b_{0},{\bf \tilde a}}$. Then by the same arguments as in Remark \ref{rem-Capelli}, it is easy to see that $R_{a_{0},b_{0},{\bf \tilde a}}= \Delta_{0}(x)^{a_{0}}R_{0,0,{\bf \tilde a}}{\Delta_{0}^*(\partial)}^{b_{0}}=X^{a_{0}}R_{0,0,{\bf \tilde a}}Y^{b_{0}}$. The elements $X^{a_{0}}R_{0,0,{\bf \tilde a}}Y^{b_{0}}$ $(a_{0},b_{0}\in \N, {\bf \tilde a} \in \N^r)$ form a vector basis of $D(V)^{G'}$. Remark now that $R_{0,0,{\bf \tilde a}}$ is in $D(V)^G={\cal T}_{0}$. Then from Proposition \ref{prop-stabilite-tau}, we get $X^{a_{0}}R_{0,0,{\bf \tilde a}}Y^{b_{0}}= \tau^{a_{0}}(R_{0,0,{\bf \tilde a}})X^{a_{0}}Y^{b_{0}}$ and $\tau^{a_{0}}(R_{0,0,{\bf \tilde a}})\in {\cal T}_{0}$. If now $a_{0}\leq b_{0}$, then $X^{a_{0}}R_{0,0,{\bf \tilde a}}Y^{b_{0}}= RY^{b_{0}-a_{0}}$, where  $R=\tau^{a_{0}}(R_{0,0,{\bf \tilde a}})X^{a_{0}}Y^{a_{0}}\in {\cal T}_{0}$. If $a_{0}> b_{0}$, then $X^{a_{0}}R_{0,0,{\bf \tilde a}}Y^{b_{0}}=RX^{a_{0}-b_{0}}$ where $R=\tau^{a_{0}}(R_{0,0,{\bf \tilde a}})\tau^{a_{0}-b_{0}}(X^{b_{0}}Y^{b_{0}})\in {\cal T}_{0}$. The first decomposition in assertion $1)$ is proved. The second decomposition is a consequence of relations $(3-1-4)$ and $(3-1-5)$.

$2)$ A slight extension of $(2-2-2)$ shows that $\C[{\cal O}]\otimes \C[V^*]$ is $G$-isomorphic to $D({\cal O})$ through the map $\varphi \otimes f \longmapsto \varphi f(\partial)$. Then the same proof as in $1)$ above shows that the elements $X^{a_{0}}R_{0,0,{\bf \tilde a}}Y^{b_{0}}$ ($a_{0}\in \Z, b_{0}\in \N, \tilde {\bf a}\in \N^r$) form a vector basis of  $D({\cal O})^{G'}={\cal T}$. Consider now an element $D\in {\cal T}$ such that $[E,D]=0$. Then necessarily $D$ is a linear combination of elements of the  form $X^{a_{0}}R_{0,0,{\bf \tilde a}}Y^{a_{0}}$ with $a_{0}\in \N$. Then, as announced previously, the two definitions of ${\cal T}_{0}$ coincide (${\cal T}_{0}=D(V)^G$ and ${\cal T}_{0}=\{D\in {\cal T},\, [E,D]=0\}$). Now if $D\in {\cal T}_{p}$, then $D=DX^{-p}X^p$ and $DX^{-p}\in {\cal T}_{0}$. Hence ${\cal T}_{p}={\cal T}_{0}X^p=X^p {\cal T}_{0}$.

Assertion $3)$ is then  obvious.

\end{proof}
\begin{rem}The inclusion $D(V)^{G'}\subset D( {\cal O})^{G'}$ is obviously strict ($X^{-1}\in D({\cal O})^{G'}\setminus D(V)^{G'}$), but the preceding results shows that these two graded algebras have the same "positive part" ($\oplus_{p\in \N}{\cal T}_{0}X^p$). 
\end{rem}
The following proposition whose proof is straightforward shows that all the Bernstein-Sato polynomials are known if one knows the Bernstein-Sato polynomials of $Y$ and of the elements of ${\cal T}_{0}$.
\begin{prop}\label{prop-calcul-Bernstein-Sato}\hfill

Let $D=D_{0}X^n$ $(n\in {\bb Z})$, resp. $D=D_{0}Y^n$ $(n\in {\bb N}^*)$, $D_{0}\in {\cal T}_{0}$, be generic homogeneous elements in ${\cal T}={\cal T}_{0}[X,X^{-1}]$ or ${\cal T}_{0}[X,Y]$.
Then 

$b_{D}({\bf a})=b_{D_{0}}({\bf a}+n)$, resp. $b_{D}({\bf a})=b_{D_{0}}({\bf a}-n)b_{Y}({\bf a})b_{Y}({\bf a}-1)\dots b_{Y}({\bf a}-n+1).$
\end{prop}





\vskip 15pt
\label{Harish-Chandra}  \subsection {The Harish-Chandra isomorphism and the center of ${\cal T}$}\hfill
\vskip 5pt

The aim of this subsection is to describe ${\cal T}_{0}=D(V)^G$ as a module over the center of ${\cal T}$. For this we will use the Harish-Chandra isomorphism for $MF$ spaces due to F. Knop.

 \vskip 5pt
Let $(G,V)$ be a $MF$ space with a one dimensional quotient. Let $B$ be a fixed Borel subgroup of $G$. Remember that $(B,V)$ is a $PV$. Recall also that we denote by $\Delta_{0},\Delta_{1},\dots,\Delta_{r}$ the set of fundamental relative invariants of $(B,V)$ and that  $\Delta_{0}$ is the unique fundamental relative invariant under $G$. We denote by $d_{i}$ (resp. $\lambda _{i}$) the degree (resp. the infinitesimal character) of $\Delta_{i}$. Let ${\go b}$ be the Lie algebra of $B$, let ${\go t}\subset {\go b}$ be a Cartan subalgebra of ${\go g}$ and let $\Sigma $ be the  set of roots of the pair $({\go g}, {\go t})$. Denote by $W$ the Weyl group of $\Sigma $. Denote by $\Sigma^+ $  the set of positive roots such that ${\go b}={\go t}+\sum _{\alpha \in \Sigma^+ }{\go g}^\alpha$.  Let $\rho=\frac{1}{2}\sum_{\alpha \in \Sigma^+ } \alpha$.
We define 
$${\go a}^*=\oplus_{i=0}^r\C \lambda _{i}\subset {\go t}^* \text{ and } A={\go a}^*+\rho \subset  {\go t}^*.$$

Let ${\cal Z}({\go g})$ be the center of the enveloping algebra of ${\go g}$. Denote by $\C[{\go t}^*]^W$ the $W$-invariant polynomials on ${\go t}^*$. One knows that the classical Harish-Chandra isomorphism is an isomorphism $H: {\cal Z}({\go g})\longrightarrow \C[{\go t}^*]^W$ which can be computed the following way. For any $\lambda \in {\go t}^*$, let $V_{\lambda}$ be the irreducible highest weight module with highest weight $\lambda$. It is well  known that ${\cal Z}({\go g})$ acts by scalar multiplication on $V_{\lambda}$. The scalar by which an element $z\in {\cal Z}({\go g})$ acts on $V_{\lambda}$ is precisely $H(z)(\lambda + \rho)$.

The natural representation of $G$ on $\C[V]$ extends to a representation of the enveloping algebra ${\cal U}({\go g})$ on the same space $\C[V]$.   Hence  $z\in {\cal Z}({\go g})$ acts on $V_{\bf a}$ by the scalar $H(z)(-\lambda_{\bf a}+\rho)\renewcommand{\thefootnote}{\fnsymbol{footnote}}\footnote{The change of sign   is due to the fact that we consider here characters of relative invariants instead of highest weights.}$ where $\lambda_{\bf a}= \sum_{i=0}^r a_{i}\lambda_{i}$ (remember that ${\bf a}=(a_{0},\dots,a_{r})$). Conversely if $\lambda=a_{0}\lambda_{0}+\dots+a_{r}\lambda_{r}$ we define ${\bf a}_{\lambda}=(a_{0},\dots,a_{r})\in \C^{r+1}$ . By abuse of notation if $b_{D}$ is the Bernstein-Sato polynomial of $D\in {\cal T}_{0}$, we set $b_{D}(\lambda)=b_{D}({\bf a}_{\lambda})$. Hence $H(z)(-\lambda+\rho)=b_{D}(\lambda).$

On the other hand any $D\in D(V)^G={\cal T}_{0}$ acts on each $V_{\bf a}$ by the scalar $b_{D}({\bf a})$, where $b_{D}({\bf a})$ is the Bernstein-Sato polynomial of $D$. This allows us to define the map:

$$\begin{array}{cccl}h: &D(V)^G&\longrightarrow &\C[A]\\
{ }&D&\longmapsto &h(D):-\lambda+\rho\longmapsto h(D)(-\lambda+ \rho)=b_{D}({\lambda}) 
\end{array}
$$
where $\C[A]$ denotes the algebra of polynomials on the affine space $A= {\go a}^*+\rho\subset {\go t}^*$.     

Let $\pi(z)$ be the operator in $D(V)^G$ which represents the action of $z$ on $\C[V]$ and let $r:\C[{\go t}^*]^W \longrightarrow \C[A]$ be the restriction homomorphism. It is clear from the definitions that the following diagram commutes:

$$\xymatrix{
{\cal Z}({\go g})\ar[d]_{\pi}\ar[rr]^{H}&&\C[{\go t}^*]^W\ar[d]^{r}\\
D(V)^G \ar[rr]^{h} & &\C[A] \\
}$$

\begin{theorem}\label{th-Knop}\hfill

{\rm (Knop, see \cite{Knop-2}  Th. 4.8 and Corollary 4.9
or \cite {Benson-Ratcliff-survey}, Th. 9.2.1)}\hfill

The homomorphism $h$ is injective and there exists a finite group $W_{0}$ (sometimes called the little Weyl group) which is a subgroup of the stabilizer of $A$ in $W$, such that the image of $h$  is $\C[A]^{W_{0}}$. Hence $h$ is an isomorphism between $D(V)^G$ and $\C[A]^{W_{0}}$. The isomorphism $h$ is called the Harish-Chandra isomorphism for the $MF$ space $(G,V)$. Moreover $W_{0}$ acts as a reflection group on ${\go a}^*$.

\end{theorem}

Let us see what is the automorphism of $\C[A]^{W_{0}}$ which corresponds to the action of $\tau$ on $D(V)^G$ through the Harish-Chandra isomorphism $h$. Let $D\in D(V)^G$. Then $h(\tau(D))(-\lambda+\rho)=h(XDX^{-1})(-\lambda+\rho)=b_{XDX^{-1}}(\lambda)=b_{D}(\lambda-\lambda_{0})$. This calculation proves of course that $\C[A]^{W_{0}}$ is stable under $P(\lambda+\rho)\longmapsto P((\lambda-\lambda_{0})+\rho)$. Therefore we make the following definition.
\begin{definition}\label{tau-poly} 

 By abuse of notation  $\tau$ will also denote the automorphism of $\C[A]^{W_{0}}$ which is defined by $\tau(P)(\lambda+\rho)=P((\lambda-\lambda_{0})+\rho)$ $(P\in \C[A]^{W_{0}})$. Let $\C[A]^{W_{0},\tau}$ denote the set of elements in $\C[A]^{W_{0}}$ which are invariant under $\tau$.

\end{definition}

\begin{prop}\label{prop-centre-T}\hfill

Let ${\cal Z}({\cal T})$ be the center of ${\cal T}=D({\cal O})^{G'}$. Then ${\cal Z}({\cal T})$ is also the center of ${\cal T}_{0}[X,Y]=D(V)^{G'}$. Moreover the following assertions are equivalent:

$i)$ $D\in{\cal Z}({\cal T})$

$ii)$ $ D\in {\cal T}_{0}$ and $\tau(D)=D$ $($i.e. $D$ commutes with $X$ $)$.

$iii)$ $D\in {\cal T}_{0}$ and the Bernstein-Sato polynomial $b_{D}(a_{0},a_{1}, 
\dots, a_{r})$ does not depend on $a_{0}$.

$iv)$ $ D\in {\cal T}_{0}$ and   $D$ commutes with $Y$.

$v)$ $D\in {\cal T}_{0}$ and $h(D)\in \C[A]^{W_{0},\tau}$.

\end{prop}

\begin{proof}

$i) \Rightarrow ii)$: Let $D\in{\cal Z}({\cal T})$. Then $[E,D]=0$, hence $D\in {\cal T}_{0}$, and $[D,X]=0$.

$ii) \Rightarrow iii)$:  Let $D\in{\cal T}_{0}$. If $XD=DX$ then, from the definitions we have  $b_{XD}(a_{0},a_{1}, \dots, a_{r})=b_{D}(a_{0},a_{1}, \dots, a_{r})=b_{DX}(a_{0},a_{1}, \dots, a_{r})=b_{D}(a_{0}+1,a_{1}, \dots, a_{r})$, hence $b_{D}(a_{0},a_{1}, \dots, a_{r})$ does not depend on $a_{0}.$

$iii) \Rightarrow i)$: Suppose that for $D\in {\cal T}_{0}$, the Bernstein-Sato polynomial does not depend on $a_{0}$. Then the elements $XD$ and $DX$ in ${\cal T}_{1}$ have the same Bernstein-Sato polynomial. Hence $XD=DX$ (Lemma \ref{lemma-Bernstein-Sato}). Then from Proposition \ref {prop-poly-XY-graduation} 2) we see that $D\in {\cal Z}({\cal T})$.

$iii) \Rightarrow iv)$: Let $D\in {\cal T}_{0}$ such that $b_{D}$ does not depend on  $a_{0}$. Then $b_{DY}(a_{0},a_{1},\dots,a_{r})=b_{D}(a_{0}-1,a_{1},\dots,a_{r})b_{Y}(a_{0},a_{1},\dots,a_{r})=b_{D}(a_{0},a_{1},\dots,a_{r})b_{Y}(a_{0},a_{1},\dots,a_{r})=b_{YD}(a_{0},a_{1},\dots,a_{r})$. Hence $DY=YD$.

$iv) \Rightarrow iii)$: If $DY=YD$, then $b_{DY}(a_{0},a_{1},\dots,a_{r})=b_{D(a_{0}-1,a_{1},\dots,a_{r})}b_{Y}(a_{0},a_{1},\dots,a_{r})=b_{YD}(a_{0},a_{1},\dots,a_{r})=b_{Y}(a_{0},a_{1},\dots,a_{r})b_{D}(a_{0},a_{1},\dots,a_{r})$. Hence $b_{D}$ does not depend  on $a_{0}$.

The equivalence of $iii)$ and $v)$ is obvious as $h(D)(-\lambda+\rho)=b_{D}(\lambda)$.

From $ii)$ we obtain that ${\cal Z}({\cal T})$ is also the center of ${\cal T}_{0}[X,Y]$.

\end{proof}

\begin{rem}\label{rem-comX-comY}

As a consequence of the preceding Proposition it is worthwhile noticing that if $D\in D(V)^{G'}$ (or $D\in {\cal T}$)  commutes with two operators among $(X,E,Y)$, then $D$ commutes with the third one. This is a well known  property if $(X,E,Y)$ is an ${\go {sl}}_{2}$-triple. But we know from \cite{Igusa} that except if  $\Delta_{0}$ is  quadratic or linear the Lie algebra generated by $(X,E,Y)$ is infinite dimensional. We will see that the associative algebra generated by $(X,E,Y)$ over ${\cal Z}({\cal T})$ is "similar" to ${\cal U}({\go {sl}}_{2}({\cal Z}({\cal T}))$, see Theorem \ref {th-generateurs-relations}  below.
\end{rem}

 Define a linear form $\mu$ on ${\go a}^*$ by
$$\mu(a_{0}\lambda_{0}+\dots+a_{r}\lambda_{r})=\sum_{i=0}^r a_{i}d_{i}=b_{E}({\bf a})\quad ({\bf a}=(a_{0},\dots,a_{r})\in \C^{r+1})$$
($\mu$ is the {\it degree form}, as its value on  ${\bf a}=(a_{0},\dots,a_{r})\in \N^{r+1}$ is equal to the degree of the polynomials in $V_{\bf a}$.  
Define also 
$${\cal M}=\{ \lambda\in {\go a}^*\,|\, \mu(\lambda)=0\} \text{ and } M= {\cal M}+\rho\subset A.$$ 
Note that $M=\{\lambda+\rho\in A\,|\, h(E)(-\lambda+\rho)=0\}$. As $h(E)$ is $W_{0}$-invariant, so is  the set $M$. Set 
$$I(M)= \{P\in \C[A]^{W_{0}}\,|\, P_{|_{M}}=0\}.$$
\vskip 5pt
The key lemma is the following.

\begin{lemma}\label{lemma-key} \hfill

 We have $I(M)=\C[A]^{W_{0}}h(E)$ and 
$$\C[A]^{W_{0}}=\C[A]^{W_{0},\tau}\oplus I(M).$$  
\end{lemma}

\begin{proof}

Let $P\in I(M)$.  It is a polynomial on the affine subspace $A\subset {\go t}^*$ which vanishes on $M$ which is the set of zeros of the irreducible polynomial $h(E)$. Therefore $P=h(E)Q$. As $P$ and $h(E)$ are $
W_{0}$-invariant, so is also the polynomial $Q$. Hence $I(M)\subset\C[A]^{W_{0}}h(E)$. As the reverse inclusion is obvious we get $I(M)=\C[A]^{W_{0}}h(E)$.

Let $F=\C\lambda_{0}\subset {\go a}^*$. As obviously ${\go a}^*={\cal M}\oplus F$, we have $A=M\oplus F$.    Remember that ${\go t}={\go c}\oplus {\go t}'$ where  ${\go c}$ is the center of ${\go g}$. The infinitesimal character $\lambda_{0}$ is a character of ${\go g}$, and is therefore trivial on ${\go t}'\subset{\go g}'$.  As any $w_{0}\in W_{0}$ fixes pointwise the center ${\go c}$ of ${\go g}$, we see that $F$ is pointwise fixed by $W_{0}$.

Let $Q\in \C[M]^{W_{0}}$. Define
$$\widetilde{Q}(m+f)=Q(m), \,\,\, \text{ for all } m\in M, f\in F.$$
From the preceding discussion we obtain that $\widetilde{Q}$ is $W_{0}$-invariant, in other words $\widetilde{Q}\in \C[A]^{W_{0}}$. But in fact $\widetilde{Q}$ is also $\tau$-invariant: $\tau(\widetilde{Q})(m+ f)=\widetilde{Q}(m+f-\lambda_{0})=Q(m)=\widetilde{Q}(m+f)$. Hence $\widetilde{Q}\in \C[A]^{W_{0},\tau}$, in other words any $W_{0}$-invariant polynomial on $M$ can be extended to a $(W_{0},\tau)$-invariant polynomial on $A$. This extension is in fact unique: for any $\tau$-invariant extension $\widetilde{\widetilde{Q}}$ of $Q$ we have $\widetilde{\widetilde{Q}}(m+x\lambda_{0})= \widetilde{\widetilde{Q}}(m+(x+1)\lambda_{0})$ and hence $\widetilde{\widetilde{Q}}=\widetilde{Q}$. Hence we have proved that the restriction map:
$$\begin{array}{rcl}\C[A]^{W_{0},\tau}&\longrightarrow& \C[M]^{W_{0}}\\
P&\longmapsto&P_{|_{M}}
\end{array}$$
is bijective (and therefore 	$\C[A]^{W_{0},\tau}\cap I(M)= \{0\}$) and the inverse map is $Q\longmapsto \widetilde Q$. Now for $P\in \C[A]^{W_{0}}$ we can write:
$$P=\widetilde{P_{|_{M}}} +(P-\widetilde{P_{|_{M}}}).$$
From the  discussion above we have $\widetilde{P_{|_{M}}} \in \C[A]^{W_{0},\tau}$, and  $(P-\widetilde{P_{|_{M}}})\in I(M)$.

\end{proof}

\begin{theorem}\label{th-key-structure-T0}\hfill

$1)$ ${\cal T}_{0}=D(V)^G={\cal Z}({\cal T})\oplus E{\cal T}_{0}$

$2)$ Any element $H\in D(V)^G$ can be uniquely  written in the form 
$$H= H_{0}+EH_{1}+E^2H_{2}\dots+E^kH_{k}$$
where $H_{i}\in {\cal Z}({\cal T})$, $i=1,2,\dots,k \in \N$.
\end{theorem}

\begin{proof} 

Through the Harish-Chandra isomorphism $h$, the algebra $D(V)^G={\cal T}_{0}$ corresponds to $\C[A]^{W_{0}}$, the algebra ${\cal Z}({\cal T})$ corresponds to  $\C[A]^{W_{0},\tau}$ and the ideal  $E{\cal T}_{0}$ corresponds to $I(M)$. Therefore the first assertion is just the pull back by $h$ of the decomposition obtained in Lemma \ref{lemma-key}.

An element $H\in D(V)^G$ can therefore be uniquely written $H=H_{0}+EH^1$, with $H_{0} \in {\cal Z}({\cal T})$, and $H^1 \in  {\cal T}_{0}$. By induction we obtain a  decomposition $H= H_{0}+EH_{1}+E^2H_{2}\dots+E^{k-1}H_{k-1}+E^kH^{k}$ where $H_{0},\dots,H_{k-1}\in {\cal Z}({\cal T})$, and $H^k\in {\cal T}_{0}$. The process stops because if $k$ is greater than  the degree in $a_{0}$ of $b_{H}$, then necessarily $H^k=0$ (see Proposition \ref{prop-centre-T}).

\end{proof}

From the preceding Theorem and Proposition \ref{prop-poly-XY-graduation} we obtain immediately the following Corollary.

\begin{cor}\label{cor-decomp-D(V)G'}\hfill

 $1)$ Let $D\in {\cal T} $, then $D$ can be  written uniquely  in the form:
$$D=\sum_{k\in {\bb Z}, \ell \in {\bb N}}H_{k,\ell}E^{\ell}X^{k} \text{  or  } D=\sum_{k\in {\bb Z}, \ell \in {\bb N}}H_{k,\ell} X^{k}E^{\ell} \text{  \rm (finite \,\,sums)}$$
where $H_{k,\ell}\in {\cal   Z}( {\cal   T})$

 $2)$ Let $D\in {\cal T}_{0}[X,Y] $, then $D$ can be   written uniquely  in the form:
$$\begin{array}{ccl}D=&\sum_{k\in {\bb N}^*, \ell \in {\bb N}}H_{k,\ell}E^{\ell}Y^{k}&+\sum_{r\in {\bb N}, s \in {\bb N}}H'_{r,s}E^{s}X^{r} \text{   {\rm (finite \,\,sum)}   or  }\\
D=&\sum_{k\in {\bb N}^*, \ell \in {\bb N}}H_{k,\ell} Y^{k}E^{\ell}&+\sum_{r\in {\bb N}, s \in {\bb N}}H'_{r,s} X^{r}E^{s}  \text{   {\rm (finite \,\,sum)}}
\end{array}$$
where $H_{k,\ell}, H'_{r,s}\in {\cal   Z}( {\cal   T})$

\end{cor}

\begin{cor}\label{cor-decomp-pol-sym}
Let $P\in { \bb C}[A]^{W_{0} }  $. Then $P$ can be uniquely written in the form
$$P(-\lambda+\rho)=\sum_{i=0}^p\alpha_{i}(-\lambda+\rho)(a_{0}d_{0}+a_{1}d_{1}+\dots+a_{r}d_{r})^i$$
where $\alpha_{i}\in { \bb C}[A]^{W_{0},\tau}$ and where $\lambda=a_{0}\lambda_{0}+a_{1}\lambda_{1}+\dots+a_{r}\lambda_{r}\in \go{a}^*$.
\end{cor}

\begin{proof} 

 As $h(E)(-\lambda+\rho)=a_{0}d_{0}+a_{1}d_{1}+\dots+a_{r}d_{r}$, the preceding  decomposition  is just the image through the Harish-Chandra isomorphism  of the decomposition in Theorem 
 \ref{th-key-structure-T0} $2)$.

\end{proof}

\begin{rem}\label{rem-action-affine}Let us make some remarks about the $W_{0}$-action on $A$. In fact it is easy to see that as $W_{0}$ stabilizes the affine space $A={\go a}^*+\rho$ it also stabilizes ${\go a}^*$ (this is implicit in Theorem \ref{th-Knop}). Moreover if we denote by $0_{\rho}$ the barycenter of the $W_{0}$-orbit  of $\rho$, then $0_{\rho}$ is a fixed point of the $W_{0}$-action on $A$  which is in $M$. As $\C[A]^{W_{0}}=\C[{\go a}^*+\rho]^{W_{0}}=\C[{\go a}^*+0_{\rho}]^{W_{0}}\simeq \C[{\go a}^*]^{W_{0}}$, and as $ {\cal T}_{0}=D(V)^G\simeq \C[A]^{W_{0}}$ is a polynomial algebra in $r+1$ variables by Theorem \ref {th-generateurs-op-inv}, the group $W_{0}$ acts as a reflection group on ${\go a}^*$ by the Shephard-Todd-Chevalley Theorem (this is a part of  Knop's argument for Theorem \ref{th-Knop}). Hence by the Theorem of Chevalley, the $r+1$ algebraically independent generators of the algebra $\C[A]^{W_{0}}\simeq \C[{\go a}^*]^{W_{0}}$   can be chosen to be homogeneous, either as functions on the vector space ${\go a}^*$, or as functions on $A$, for the vector space structure on $A$  defined by taking $0_{\rho}$ as origin. 

\end{rem}
\vskip 5pt

We will now describe more precisely the algebra $ {\cal   Z}( {\cal   T})$.  
\begin{theorem}\label{th-Z(T)}\hfill

$1)$  $ {\cal   Z}( {\cal   T})$ is a polynomial algebra in $r$ variables. For $D\in {\cal T}_{0}$, let us denote by $\overline{D}$ the projection of $D$ on   $ {\cal   Z}( {\cal   T})$ according to the decomposition  ${\cal T}_{0}={\cal Z}({\cal T})\oplus E{\cal T}_{0}$.    Remember  from Theorem \ref{th-generateurs-op-inv}  that the set  $R_{0},\dots,R_ {r-1},R_{r}$ of  Capelli operators associated to the invariants $\Delta_{0}, \Delta_{1}, \dots, \Delta_{r}$ ordered by decreasing degree is   a set algebraically independent generators of ${\cal T}_{0}$.  Then  $\{\overline{{R}_{0}},\dots,\overline{{R}_{r-1}}\}$  is a set of algebraically independent generators of $ {\cal   Z}( {\cal   T})$.  

$2)$ Let $D$ be an element of ${\cal T}_{0}$ and let $b_{D}$ be its Bernstein-Sato polynomial. Then the Bernstein-Sato polynomial of $\overline{D}$ is given  by 
$$b_{\overline{D}}(a_{0},a_{1},\dots,a_{r})=b_{D}(-\frac{(a_{1}d_{1}+\dots+a_{r}d_{r})}{d_{0}}, a_{1},\dots,a_{r}).$$
 
 \end{theorem}
 
 \begin{proof}

 $1)$ Let us remark first that ${\cal Z}({\cal T})$ is already known to be a polynomial algebra from a result of Knop (\cite{Knop-1}). He has  proved that for a regular action of a reductive group on a smooth affine variety the center of the ring of invariant differential operators is always a polynomial algebra. We give here a direct proof and obtain some extra information. We know from Proposition \ref{prop-centre-T} that ${\cal Z}({\cal T})$ is isomorphic, through the Harish-Chandra isomorphism $h$, to $\C[A]^{W_{0},\tau}$. From the proof of Lemma \ref {lemma-key} we know that $W_{0}$ stabilizes $M$ and that  $\C[A]^{W_{0},\tau}\simeq \C[M]^{W_{0}}= (\C[A]^{W_{0}})_{|_{M}}$ .  As $W_{0}$ is a reflection group on $A$ (this means that it is generated by the reflections it contains), so is ${W_{0}}_{|_{M}}$. Therefore $\C[M]^{W_{0}}$ (and hence ${\cal Z}({\cal T})$)  is a polynomial algebra in $r=\dim M$ variables by Chevalley's Theorem. We know from Remark  \ref{rem-Capelli} d) that  $\{R_ {0},\dots,R_ {r-1},E\}$ is also a  set  algebraically independent generators of ${\cal T}_{0}$, hence $\{h(R_ {0}),\dots,h(R_ {r-1}),h(E)\}$ is a set of algebraically independent generators of $\C[A]^{W_{0}}$. As $h(E)_{|_{M}}=0$ we obtain that $\C[M]^{W_{0}}=\C[h(R_ {0})_{|_{M}},\dots,h(R_ {r-1})_{|_{M}}]$. As the transcendence degree of $Frac(\C[M]^{W_{0}})$ over $\C$ is $r$, the generators $h(R_ {0})_{|_{M}},\dots,h(R_ {r-1})_{|_{M}}$ are algebraically independent. Taking their inverse image under $h$ gives the first assertion of the Theorem. 
 
 $2)$ As we have seen  the decomposition ${\cal T}_{0}={\cal Z}({\cal T})\oplus E{\cal T}_{0}$ is nothing else but the inverse image under $h$ of the decomposition $\C[A]^{W_{0}}=\C[A]^{W_{0},\tau}\oplus I(M)$. Let $D\in {\cal T}_{0}$. From the proof of  Lemma \ref {lemma-key} we have $h(\overline{D})=\widetilde{h(D)_{|_{M}}}$, where $\widetilde{h(D)_{|_{M}}}$ is the unique $(W_{0},\tau)$-invariant extension to $A$ of $h(D)_{|_{M}}$. For $\lambda=a_{0}\lambda_{0}+\dots+a_{r}\lambda_{r} \in {\go a}^*$, we have $h(E)(\lambda+\rho)=b_{E}(-\lambda)=-(a_{0}d_{0}+\dots+a_{r}d_{r})=-\mu(\lambda)$ (the  degree form). Remember also that ${\go a}^*={\cal M}\oplus F$, where $F=\C\lambda_{0}$, where ${\cal M}=\ker(\mu)$. Let us write $\lambda=m_{\lambda}+ \alpha\lambda_{0}$, according to this decomposition. Then $b_{E}(\lambda)=\alpha b_{E}(\lambda_{0})=\alpha d_{0}$. Hence $\alpha= \frac{\mu(\lambda)}{d_{0}}$ and $m_{\lambda}=\lambda-\frac{\mu(\lambda)}{d_{0}}\lambda_{0}$. Then we obtain:
 $$\begin{array}{rl}
 b_{\overline{D}}(\lambda)=h({\overline{D}})(-\lambda+\rho)&=\widetilde{h(D)_{|_{M}}}(-\lambda+\rho)= \widetilde{h(D)_{|_{M}}}(-\lambda+\frac{\mu(\lambda)}{d_{0}}\lambda_{0}-\frac{\mu(\lambda)}{d_{0}}\lambda_{0}+\rho)\\
 \\
 {}&=\widetilde{h(D)_{|_{M}}}(-\lambda+\frac{\mu(\lambda)}{d_{0}}\lambda_{0}+\rho)\\
 \\
 {}& =h(D)_{|_{M}}(-\lambda+\frac{\mu(\lambda)}{d_{0}}\lambda_{0}+\rho)=h(D)(-\lambda+\frac{\mu(\lambda)}{d_{0}}\lambda_{0}+\rho)\\
 \\
{}&=b_{D}(\lambda-\frac{\mu(\lambda)}{d_{0}}\lambda_{0}).
 \end{array}$$
 If we translate this into the $(a_{0},\dots,a_{r})$-variables we obtain   the second assertion.
 
 \end{proof}
 
\begin{cor}\label{cor-centre}\hfill

Let $b_{Y}$ be the Bernstein-Sato operator of $Y$. For any $\ell \in \N$ the element of $End(\C[V])$ wich acts on each space $V_{\bf a}$ as the scalar multiplication by
$b_{Y}(-\frac{(a_{1}d_{1}+\dots+a_{r}d_{r})}{d_{0}}+\ell, a_{1},\dots,a_{r})$
is the differential operator $\overline{X^{1-\ell}YX^{\ell}}\in {\cal Z}({\cal T})$.  Moreover, if $(G,V^+)$ is a $PV$ of commutative parabolic type, the differential operators $\overline{X^{1-\ell}YX^{\ell}}$ ($\ell=0,1,\dots,r$) are generators of ${\cal Z}({\cal T})$.

\end{cor}

\begin{proof} 

As  $b_{X^{1-\ell}YX^{\ell}}(a_{0},\dots,a_{r})= b_{Y}(a_{0}+\ell,a_{1},\dots,a_{r})$, the first assertion is an immediate consequence of Theorem \ref{th-Z(T)}.  If $(G,V^+)$ is a PV of commutative parabolic type, we knom from Theorem  \ref{th-special-commutatif} that the operators $X^{1-\ell}YX^{\ell}$ ($\ell = 0,\dots,r$) are (algebraically independent) generators of ${\cal T}_{0}$. 
 
\end{proof}

\vskip 15pt
\subsection{The  case of regular PV's of commutative parabolic type}\hfill
 \vskip 5pt

In the case where $(G,V^+)$ is a regular $PV$ of commutative parabolic type (see Example \ref{ex-commutative-PV}), we obtain some specific results.  


 \begin{theorem}\label{th-special-commutatif} \hfill
 
 Let $(G,V^+)$ be a regular $PV$
  of commutative parabolic type.
  
  $1)$ The degree of $\Delta_{0}$ is equal to $r+1$ which is the rank of $(G,V^+)$ as a $MF$ space. More generally the degree of $\Delta_{i} $ is equal to $r+1-i$.
  
  $2)$ For $\ell \in \Z$ set $D_{\ell}=X^{1-\ell}YX^{\ell}$. Then $D_{0},D_{1},\dots,D_{r}$ are algebraically independent generators of ${\cal T}_{0}=D(V^+)^G$ $($i.e ${\cal T}_{0}=\C[D_{0},D_{1},\dots,D_{r}]$$)$.
  
  $3)$ We have ${\cal T}=D(\Omega^+)^{G'}=\C[X,X^{-1},Y]$ where $\C[X,X^{-1},Y]$ is the associative subalgebra of $D(\Omega^+)$ generated by $X,X^{-1},Y$.    
 
  $4)$ We have ${\cal T}_{0}[X,Y]=D(V)^{G'}= \C[X,Y,R_{1},\dots,R_{r}]$ where the $R_{i}$'s are the Capelli operators which were introduced before Theorem \ref{th-generateurs-op-inv}, and where $\C[X,Y,R_{1},\dots,R_{r}]$ is the associative subalgebra of $D(V^+)$ generated by $X,Y,$ $R_{1},\dots,R_{r}$.  \end{theorem}
 
 \begin{proof} 
 
 1) This first assertion is proved in \cite{Muller-Rubenthaler-Schiffmann} (Prop. 2.16 and Lemme 3.7).
 
 2) We need now to use some technical results from the structure theory of commutative $PV$'s of parabolic type. For details see \cite{Muller-Rubenthaler-Schiffmann} and \cite{Rubenthaler-Schiffmann-1}. We need also results  concerning the symmetric space structure of the open $G$ orbit $\Omega^+$ in $V^+$, these can be found in \cite{Bopp-Rubenthaler-complexe}.  Let ${\go t}$ be a Cartan subalgebra of ${\go g}$, then ${\go t}$ is also a Cartan subalgebra of $\widetilde{\go g}$ (see the notation in Example \ref{ex-commutative-PV}) and let $\widetilde \Sigma$ (resp. $\Sigma$) be the root system of $(\widetilde{\go g}, {\go t})$ (resp. $({\go g}, {\go t})$). We choose an order on $\widetilde \Sigma$ such that the roots occurring in $V^+$ are positive. We know from Prop. 2.9. in \cite{Bopp-Rubenthaler-complexe} that the open $G$-orbit $\Omega^+=\{x\in V^+\,|\, \Delta_{0}(x)\neq0\}$ is a symmetric space $G/H$ where $H$ is the isotropy subgroup of a point $I^+\in \Omega^+$. The choice of $I^+$ can be made the following way. It is known that any maximal set of stronly orthogonal long roots occurring in  $V^+$ has $r+1= rk(G,V^+)$ elements. There is a canonical way to construct such a maximal set, called the "descent", see \cite{Muller-Rubenthaler-Schiffmann}, Th. 2.7. p.101.  If $\{\alpha_{0}, \alpha_{1},\dots, \alpha_{r}\}$ is such a maximal set of strongly orthogonal long roots, then the element $I^+=X_{\alpha_{0}}+X_{\alpha_{1}}+\dots+X_{\alpha_{r}}$ is generic (here as usual the $X_{\alpha_{i}}$'s are non zero root vectors). Let ${\go h}= Z_{\go{g}}(I^+)$  be the Lie algebra of $H$, and let ${\go q}$ be the   orthogonal complement of $\go{h}$ in $\go{g}$ with respect to the Killing form of $\widetilde{\go {g}}$. Let $H_{\alpha_{i}}\in \go{t}$ be the co-root of $\alpha_{i}$. Set $\go{a}=\sum_{i=0}^{r}\C H_{\alpha_{i}}$.  Then $\go{a}$ is a maximal abelian subspace of $\go{q}$ (\cite{Bopp-Rubenthaler-complexe}, Prop. 5.4) and the dual space $\go{a}^*$ can be identified with the space of restrictions of the fundamental characters $\lambda_{0},\lambda_{1},\dots,\lambda_{r}$ (\cite{Bopp-Rubenthaler-complexe}, Lemme 2.5).    Hence this definition of ${\go a}^*$ is coherent with the direct definition of $\go{a}^*$  given in section 3.2 in the general case  (${\go a}^*= \sum _{i=0}^r \C \lambda_{i}$).

 For $\lambda\in {\go t}^*$, we will denote by $\overline{\lambda}$ the restriction of $\lambda$ to ${\go a}$. Through  the "classical" Harish-Chandra isomorphism $\gamma$ for symmetric spaces (\cite{Heckmann-Schlichtkrull}, Part II, Theorem 4.3) the algebra ${\cal T}_{0}$ is isomorphic to $S({\go a})^{W_{R}}=\C[{\go a}^*]^{W_{R}}$, where $W_{R}$  is the Weyl group of the root  system $R$ of $({\go g},{\go a})$. This root system is known to be of type $A_{r}$ (the proof is the same as for Theorem 3.11 in \cite{Bopp-Rubenthaler-reel}). Hence $W_{R}$ is the symmetric group of $r+1$ variables and it acts by permutations on the $\overline{\alpha_{i}}$'s.  We will choose an order on $R$ such that $\overline{\Sigma^+}\subset R^+ $. As in  \cite{Muller-Rubenthaler-Schiffmann} and \cite{Rubenthaler-Schiffmann-1} we consider here relative invariants  $\Delta_{0}, \Delta_{1},\dots,\Delta_{r}$ with respect to the Borel subgroup defined by $\Sigma^-$. 
 Define $\rho=\frac{1}{2}\sum_{\beta\in R^-}\beta$. It is well known that for $D\in {\cal T}_{0}$ and $\lambda=\sum_{i=0}^r a_{i}\lambda_{i}\in {\go a}^*$,    $\gamma(D)(-\overline{\lambda}+\rho)$ is equal to the eigenvalue of  $D$ acting on $\Delta_{0}^{a_{0}}\dots\Delta_{r}^{a_{r}}$. In other words we have:
 $$\gamma(D)(-\overline{\lambda}+\rho)=b_{D}(\lambda).$$
 From \cite{Rubenthaler-Schiffmann-2}, Lemme 3.9 p. 155 we know that 
 $$\rho={d\over4}\sum_{i<j}(\overline{\alpha_i}-\overline{\alpha_j})={d\over4}\sum_{i=0}^r(r-2i)\overline{\alpha_i}$$
 and from  [ibid.], Lemme 3.8 p. 155 we also have:
 $$\overline{\lambda}=
a_0\overline{\alpha_0}+(a_0+a_1)\overline{\alpha_1}+\ldots+(a_0+\cdots+a_r)\overline{\alpha_r}.\renewcommand{\thefootnote}{\fnsymbol{footnote}}\footnote{The change of sign with respect to Lemme 3.8 in \cite{Rubenthaler-Schiffmann-2} is again due to the fact that we consider here characters of relative invariants instead of highest weights.} $$
Let us now make the following change of variables:
$$s_{i}=a_{0}+\dots+a_{i}, \text{ for }i=0,\dots,r.$$
As $b_{D_{\ell}}(\lambda)=b_{Y}(s_{0}+\ell,\dots,s_{r}+\ell)= c\prod_{i=0}^r(s_{i}+\ell+i\frac{d}{2})$ (see Example \ref{example-b_Y-commutative}) we obtain  
$$\begin{array}{rl}
\gamma(D_{\ell})(\overline {\lambda})&= b_{D_{\ell}}(-\lambda+\rho)=b_{D_{\ell}}(\sum_{i=0}^r -s_{i}\overline{\alpha_{i}}+ {d\over4}\sum_{i=0}^r(r-2i))\overline{\alpha_i})\\
\\
&=c\prod_{i=0}^r(-s_{i}+\frac{d}{4}r+\ell).
\end{array}$$

As expected  the polynomials $\gamma(D_{\ell}$) are symmetric in the $s_{i}$ variables (i.e. invariant under $W_{R}$). Moreover it is easy to prove that these polynomials, for $\ell=0,\dots,r$, are algebraically independent generators of the algebra of symmetric polynomials. This proves $2)$.

$3)$ As ${\cal T}={\cal T}_{0}[X,X^{-1}]$ (see Proposition \ref{prop-poly-XY-graduation}), and as, from $2)$, the elements of ${\cal T}_{0}$ are polynomials in $X,X^{-1},Y$ we obtain that ${\cal T}\subset \C[X,X^{-1},Y]$. The inverse inclusion is obvious.

$4)$ The inclusion  $ \C[X,Y,R_{1},\dots,R_{r}]\subset D(V^+)^{G'}= {\cal T}_{0}[X,Y]$ is obvious. Conversely, from Theorem \ref{th-generateurs-op-inv} we have ${\cal T}_{0}[X,Y]= \C[R_{0},R_{1},\dots,R_{r}][X,Y]$. As $ R_{0}=XY$ (see Remark \ref{rem-Capelli}), we have ${\cal T}_{0}[X,Y]\subset  \C[X,Y,R_{1},\dots,R_{r}]$.
 
 \end{proof}
 
 \begin{rem} According to Terras \cite{Terras}(II, p.208), the operators $D_{\ell}$ were first considered by Selberg on positive definite symmetric matrices. They appear also in Maass (\cite{Maass}),  in the same  context of positive definite symmetric matrices. In the setting  of symmetric cones, the analogue of assertion $2)$ of the preceding theorem can be found in \cite{Faraut-Koranyi-book}(Corollary XIV.1.6).
 \end{rem}
 
 \begin{rem} Note that for $PV$'s of commutative parabolic type we have $R_{r}=E$. In the special case where
$G\simeq SO(k)\times{ \bb C} ^*$ and $V^+\simeq { \bb C} ^k$, we have always  $r=1$, and assertion $4)$ of the 
preceding theorem yields
$${ D}({ \bb C} ^k)^{SO(k)}={ \bb C} [Q(x),Q(\partial),E]$$
where $Q(x)=X=\sum_{i=1}^kx_i^2$, $Q(\partial)=Y=\sum_{i=1}^k{\partial^2\over
\partial x_i^2}$.

   \noindent  This was   proved by S. Rallis and G. Schiffmann (\cite{Rallis-Schif} , Lemma
5.2. p. 112).

 \end{rem}
 \vskip 10pt

 
  \vskip 20pt
\section{The structure of $D(V)^{G'}$}\hfill

 \vskip 5pt
\subsection{Smith algebras over rings}\hfill
\vskip 5pt

As usual if $a,b$ are elements of an associative algebra we define $[a,b]=ab-ba$.

 \begin{definition}\label{def-Smith-algebra} 
 
 Let ${\bf A}$ be a commutative associative algebra over $\C$, with unit element $1$ and without zero divisors. Let $f, u\in {\bf A}[t]$ be two polynomials in one variable with coefficients in ${\bf A}$. Let $n\in \N^*$.
 
 $1)$ The  Smith algebra $S({\bf A}, f, n)$ is the associative algebra over ${\bf A}$ with generators $(x,y,e)$ subject to the relations $[e,x]=nx$, $[e,y]=-ny$, $[y,x]=f(e)$.

 $2)$ The  algebra $U({\bf A}, u, n)$ is the associative algebra over ${\bf A}$ with generators $(\tilde x, \tilde y, \tilde e)$ subject to the relations $[\tilde e,\tilde x]=n\tilde x$, $[\tilde e,\tilde y]=-n\tilde y$, $\tilde x\tilde y= u(\tilde e)$, $\tilde y\tilde x= u(\tilde e+n)$. 
  \end{definition}
  
  \begin{rem}
  
  1) The algebras $S(\C, f, n)$ were introduced and intensively studied by S. P. Smith (\cite{Smith}) who called them "algebras similar to ${\cal U}({\go{sl}}_{2})$", where ${\cal U}({\go{sl}}_{2})$ is the enveloping algebra of ${\go{sl}}_{2}$.  In fact they share   many interesting properties with   ${\cal U}({\go{sl}}_{2})$, in particular they have a very rich representation theory.

  2) One can prove, as in \cite{Smith}, that if the degree of $f$ is one and $n\neq 0$, and if the leading coefficient is invertible in ${\bf A}$, then $S({\bf A}, f, n)$ is isomorphic to the enveloping algebra ${\cal U}({\go{sl}}_{2}({\bf A}))$.
  \end{rem}
  
   Let ${\cal R}$ be a ring and let $\sigma \in \text{Aut} ({\cal R})$. Let us recall that a $\sigma$-derivation of ${\cal R}$ is an additive  map $\delta: {\cal R}\longrightarrow{\cal R} $ such  that $\delta(su)=s\delta(u)+\delta(s)\sigma(u)$. Given a $\sigma$-derivation $\delta$, the skew   polynomial    ring over ${\cal R}$ determined by $\sigma$
 and $\delta $ is the ring ${\cal R}[t, \sigma, \delta]:=\langle {\cal R}, t\rangle/\{st-t\sigma(s)-\delta (s)|s\in {\cal R}\}$, where $\langle {\cal R}, t\rangle$ stands for the ring freely generated by ${\cal R}$  and an element $t$ with the relations given by the ring structure on ${\cal R}$ (for details see \cite{macConnell-rob}, section 1.2, p.15 or \cite{Goodearl-Warfield} p.34).  

  \begin{prop}\label{prop.Smith=poly-twistee} \hfill
  
  Let ${\go b}$ the $2$-dimensional Lie algebra over  ${\bf A}$, with basis $\{\varepsilon, \alpha\}$ and relation $[\varepsilon,\alpha]=n\alpha$. Let ${\cal U}({\go b})$ be  the enveloping algebra of ${\go b}$. Define an automorphism $\sigma$ of  ${\cal U}({\go b})$ by $\sigma(\alpha)=\alpha$ and $\sigma(\varepsilon)=\varepsilon-n$ and define also a $\sigma$-derivation $\delta$ of ${\cal U}({\go b})$ by $\delta(\alpha)=f(\varepsilon)$ and $\delta(\varepsilon)=0$. Then  $S({\bf A}, f, n)\simeq {\cal U}({\go b})[t,\sigma,\delta]$.   \end{prop}

\begin{proof} The proof is almost the same as the one  given by S. P. Smith (\cite{Smith}, Prop. 1.2.). 
 The isomorphism $S({\bf A}, f, n)\simeq {\cal U}({\go b})[t,\sigma,\delta]$ is given by $e\longmapsto \varepsilon$, $x\longmapsto \alpha$ and $y\longmapsto t$.

\end{proof}

\begin{cor}\label{cor.smith-basis}\hfill 

  \noindent  $S({\bf A}, f, n) $ is a noetherian domain with ${\bf A}$-basis $\{y^{i}x^{j}e^{k},i,j,k\in{\bb N}\}$  {  \rm (}or any similar family of ordered monomials obtained by permutation of the elements $(y,x, e)${\rm )}.
 
\end{cor}
\begin{proof}(compare with  \cite{Smith}, proof of corollary 1.3 p. 288). We know from  \cite{macConnell-rob} Th.1.2.9,   that as  ${\cal U}({\go b})$ is a  noetherian domain,  so is $S({\bf A}, f, n)\simeq {\cal U}({\go b})[t,\sigma,\delta]$. Since 
\begin{align*} {\cal U}({\go b})[t,\sigma,\delta]&=  {\cal U}({\go b})\oplus  {\cal U}({\go b})t\oplus   {\cal U}({\go b})t^2\oplus  {\cal U}({\go b})t^3\oplus \dots\oplus {\cal U}({\go b})t^\ell\oplus \dots\\
&=  {\cal U}({\go b})\oplus  t{\cal U}({\go b}) \oplus   t^2{\cal U}({\go b}) \oplus  t^3{\cal U}({\go b}) \oplus \dots\oplus t^\ell{\cal U}({\go b}) \oplus \dots
\end{align*} 
(direct sums of ${\bf A}$-modules) and since the Poincar\'e-Birkhoff-Witt Theorem is still true for enveloping algebras of Lie algebras which are free over rings (see \cite{Bourbaki-1}),   the ordered monomials in $(y,x,e)$   beginning  or ending with $y$ form a basis of the algebra $S({\bf A}, f, n)$. To obtain the basis  $\{e^iy^jx^k\}$  or $\{ x^ky^je^i\}$ it suffices to replace the algebra ${\go b}$ by the algebra ${\go b}_{-}$  which is generated by $e$ and $y$.

\end{proof}

\begin{rem}\label{rem.graduation}The adjoint action of $e$ $(u\longmapsto [e,u])$ on $ S({\bf A}, f, n)$ is semi-simple and gives a decomposition of $S({\bf A}, f, n)$ into weight spaces:
$$S({\bf A}, f, n)=\oplus_{\nu\in {\bb Z}}\,S({\bf A}, f, n)^\nu$$
where $S({\bf A}, f, n)^{\nu}=\{u\in S({\bf A}, f, n), [e,u]={\nu}n u\}$. As $[e,x^{j}y^{i} e^{k}]=n(j-i)y^{i}x^{j}e^{k}$, we obtain, using Corollary \ref{cor.smith-basis}, that  the ordered monomials of the form $ x^{i}y^{i}e^{k}$ form an  ${\bf A}$-basis for $S({\bf A}, f, n)^{0}$. Moreover as $yx=xy+f(e)$, it is easy to sea that $S({\bf A}, f, n)^{0}={\bf A}[xy,e]={\bf A}[yx,e]$, where ${\bf A}[xy,e]$ (resp. ${\bf A}[yx,e]$) denotes the ${\bf A}$-subalgebra generated by $xy$ (resp. $yx$) and $e$.
\end{rem}
\vskip 10pt

The proof of the following Lemma is straightforward.
\begin{lemma}\label{lemma.f=u-u}\hfill

 Let $n\in \N^*$ and let $f\in {\bf A}[t]$. There exists an element $u\in {\bf A}[t]$, which is unique up to 
addition of  an element of ${\bf A}$, such that 
 $$f(t)=u(t+n)-u(t)\eqno (5-1-1)$$
\end{lemma}


 \vskip 5pt


\begin{prop}\label{prop.casimir}{\rm (compare with \cite{Smith}, Prop. 1.5)}\hfill

 Let $u$ be as in the preceding Lemma. Define  
$$\Omega_{1}=xy-u(e).$$
 Then  the center of $S({\bf A}, f, n)$ is ${\bf A}[\Omega_{1}]$  which is isomorphic to the polynomial algebra ${\bf A}[t]$.
 \end{prop} 
 
 \begin{proof}

  Let us now prove that $\Omega_{1}$ is central. Obviously  $\Omega_{1}$  commutes with $e$.

   From the   defining relations  of $ S({\bf A}, f, n)$ we have  $ex=x(e+n)$ and therefore, for any $k\in {\bb N}$, $e^kx=x(e+n)^k$.
 
 This implies of course that for any polynomial $P\in {\bf A}[t]$ we have
 $$P(e)x=xP(e+n) \text{ or } P(e-n)x=xP(e). \eqno (5-1-2)$$
 Similarly one proves that
 $$P(e)y=yP(e-n) \text{ or }P(e+n)y=yP(e). \eqno(5-1-3)$$
 Let us  show that $\Omega_{1}$ commutes with $x$. Using Lemma \ref{lemma.f=u-u} and $(5-1-2)$ we obtain:
 \begin{align*}
 x\Omega_{1}&=x(xy-u(e))=x^2y-xu(e)=x(yx-f(e))-xu(e)\\
 &=x(yx-u(e+n)+u(e))-xu(e)=xyx-xu(e+n) = xyx-u(e)x  \\
 &=\Omega_{1}x.
  \end{align*}
  A similar calculation using $(5-1-3)$ shows that $\Omega_{1}$ commutes also with $y$. Hence $\Omega_{1}$ belongs to the center of $S({\bf A}, f, n)$.
  
  Let now $z$ be a central element of $ S({\bf A}, f, n) $. Then $z\in S({\bf A}, f, n)^0$. We have   $ S({\bf A}, f, n)^0 = {\bf A}[xy,e]={\bf A}[\Omega_{1},e]$, and hence $z$ can be written as follows:
  $$z=\sum c_{i}(e)\Omega_{1}^{i} \qquad \text { (finite sum) }$$
  where $c_{i}(e)\in {\bf A}[e]$.
  
  We have:
  \begin{align*}
  0&=[z,x]=[\sum c_{i}(e)\Omega_{1}^{i},x]=\sum [c_{i}(e),x]\Omega_{1}^{i}\\
  &=\sum (c_{i}(e)x-xc_{i}(e))\Omega_{1}^i= \sum x(c_{i}(e+n)-c_{i}(e))\Omega_{1}^i \text{ (using  $(5-1-2))$ }\\
  &=x(\sum (c_{i}(e+n)-c_{i}(e))\Omega_{1}^i) 
 \end{align*}
 As the algebra   $S({\bf A}, f, n)$ has no zero divisors  we get:
 $$\sum (c_{i}(e+n)-c_{i}(e))\Omega_{1}^i=0$$

  As $\Omega_{1}=xy-u(e)$, we have 
  $\Omega_{1}^i=x^iy^i \text {  modulo  monomials of the form   } e^kx^py^p \text { with } p<i.$
 Then  from Corollary \ref{cor.smith-basis} above we obtain $c_{i}(e+n)-c_{i}(e)=0 , \text{  for all } i$. As the elements $e^k$ are free over ${\bf A}$ (Corollary \ref{cor.smith-basis}) we obtain from Lemma \ref{lemma.f=u-u} that $c_{i}\in {\bf  A}$, for all $i$.   
 
 \end{proof}
 
 \begin{rem}\label{rem-U=quotient} Conversely let us start with  $u\in {\bf A}[t]$.  Define $f\in {\bf A}[t]$  by  $f(t)=u(t+n)-u(t)$. Then, from the definitions we have:
$$U({\bf A},u,n)= S({\bf A}, f, n)/(  x  y-u(  e))=S({\bf A}, f, n)/(\Omega_{1})$$
where $(  x  y-u(  e))=(\Omega_{1})$ is the   ideal  generated by $   x  y-u(  e) =\Omega_{1}$. Again, as for $S({\bf A}, f, n)$,  the adjoint action of $\tilde{e}$ gives a decomposition of $U({\bf A},u,n)$ into weight spaces:
$$U({\bf A},u,n)=\oplus_{\nu\in {\bb Z}}U({\bf A},u,n)^{\nu}\eqno (5-1-4)$$
where $U({\bf A},u,n)^{\nu}=\{\tilde{v}\in U({\bf A},u,n), [\tilde{e},\tilde{v}]={\nu}n \tilde{v}\}$.
\end{rem} 
\vskip 5pt
\begin{prop}\label{prop.U-contient-pol}\hfill

Let $u\in {\bf A}[t]$ and $s\in {\bb N}$. The  ${\bf A}$-linear mappings
$$\begin{array}{rclcrcl}
\varphi: {\bf A}[t]& \longrightarrow  &U({\bf A},u,n)&\quad &\psi: {\bf A}[t]& \longrightarrow&  U({\bf A},u,n)\\
P &\longmapsto  &\varphi(P)=\tilde{x}^sP(\tilde{e})&\quad&P &\longmapsto& \psi(P)=\tilde{y}^sP(\tilde{e})
\end{array}
$$

are injective   {\rm(}in particular the subalgebra ${\bf A}[\tilde{e}]\subset U({\bf A},u,n)$ generated by $\tilde{e}$ is a polynomial algebra{\rm)}.
 
\end{prop}

\begin{proof} 

  \noindent Define $f(t)=u(t+n)-u(t)$.
Every element of $S({\bf A}, f, n)$ can be written uniquely in  the form
 $$\sum a_{k,\ell,m}e^kx^\ell y^m \qquad (a_{k,\ell,m}\in {\bf A})$$
  (Corollary \ref{cor.smith-basis}). Therefore, from Remark \ref{rem-U=quotient}, every element in $U({\bf A},u,n)$ can be written          in  the form:
   $$\sum a_{k,\ell,m}\tilde{e}^k\tilde{x}^\ell \tilde{y}^m  \qquad (a_{k,\ell,m}\in {\bf A}).$$
  Let $P(t)=\sum_{i=0}^pa_{i}t^i$, $(a_{i}\in {\bf A})$ be a polynomial  such that $\tilde{x}^sP(\tilde{e})=0$  ({\it i.e.}      
    $ P\in \ker \varphi$). As $U({\bf A},u,n)=  S({\bf A}, f, n)/(\Omega_{1})$, we see that 
   there exists $\alpha\in S({\bf A}, f, n)$ such that 
  $$x^s\sum_{i=0}^p a_{i}e^i=\alpha \Omega_{1}=\alpha(xy-u(e)).$$
  If $\alpha=\sum a_{k,\ell,m}e^kx^\ell y^m $,  using the fact that $\Omega_{1}=xy-u(e)$ is central and relation $(5-1-2)$ we get:
 \begin{align*}\ x^s\sum_{i=0}^p a_{i}e^i &= (\sum_{k,\ell,m} a_{k,\ell,m}e^kx^\ell y^m) (xy-u(e))= \sum_{k,\ell,m} a_{k,\ell,m}e^kx^\ell(xy-u(e)) y^m \\
 &= \sum_{k,\ell,m} a_{k,\ell,m}e^kx^{\ell+1} y^{m+1}- \sum_{k,\ell,m} a_{k,\ell,m}e^kx^\ell u(e) y^m\\
 &=  \sum_{k,\ell,m} a_{k,\ell,m}e^kx^{\ell+1} y^{m+1}- \sum_{k,\ell,m} a_{k,\ell,m}e^ku(e-\ell n)x^\ell   y^m  \quad  (*)
 \end{align*}
 Suppose now that $\alpha\neq 0$, then one can define 
 $$\ell_{0}=\max\{\ell \in {\bb N}, \exists k,m , a_{k,\ell,m}\neq 0\}.$$
 Let $k_{0},m_{0}$ be such that $a_{k_{0},\ell_{0},m_{0} }\neq0$.   From $(*)$ above we get
 $$x^s\sum_{i=0}^p a_{i}e^i +\sum_{k,\ell,m} a_{k,\ell,m}e^ku(e-\ell n)x^\ell   y^m =\sum_{k,\ell,m} a_{k,\ell,m}e^kx^{\ell+1} y^{m+1}.$$
 Using again $(5-1-2)$ we obtain:
 $$ \sum_{i=0}^p a_{i}(e-ns)^ix^s +\sum_{k,\ell,m} a_{k,\ell,m}e^ku(e-\ell n)x^\ell   y^m =\sum_{k,\ell,m} a_{k,\ell,m}e^kx^{\ell+1} y^{m+1}.$$
 The left hand side of the preceding equality does not contain the monomial $e^{k_{0}}x^  {\ell_{0}+1}y^{m_{0}+1}$, whereas the right hand side does. As the elements $e^kx^{\ell}y^m$ are a basis over ${\bf A}$ (Corollary \ref {cor.smith-basis}), we obtain a contradiction. Therefore  $\alpha=0$, and hence $x^s\sum_{i=0}^p a_{i}e^i=0$,  and again from Corollary \ref {cor.smith-basis}, we obtain that $a_{i}=0$ for all $i$. This proves that $\ker \varphi=\{0\}$. The proof for $\psi$ is similar.

 \end{proof}
 
 \begin{cor}\label{cor.U-basis} Every element $\tilde{u}$ in $U({\bf A},u,n)$ can be written uniquely in   the form
 $$\tilde{u}=\sum_{{\ell>0,k\geq 0}}\alpha_{k,\ell}\tilde{y}^{\ell}\tilde{e}^k+\sum_{m\geq 0,r\geq 0}\beta_{m,r}\tilde{x}^m \tilde{e}^r$$
 with $\alpha_{k,\ell}, \beta_{m,r}\in {\bf A}$.
 \end{cor}
 \begin{proof}
 We have already noticed that any element in $U({\bf A},u,n)$ can be written (in a non unique way) as a linear combination, with coefficients in ${\bf A}$, of the elements $ \tilde{x}^i\tilde{y}^j\tilde{e}^k$. 
 
 \noindent  Suppose that $i\geq j$. Then we have $ \tilde{x}^i\tilde{y}^j\tilde{e}^k=  \tilde{x}^{i-j} \tilde{x}^j\tilde{y}^j\tilde{e}^k.$
As   $\tilde{x}\tilde{y}=u(\tilde{e})$, we see that $\tilde{x}^j\tilde{y}^j=Q_{j}(\tilde{e})$, where $Q_{j}$ is a polynomial with coefficients in ${\bf A}$. Therefore $\tilde{x}^i\tilde{y}^j\tilde{e}^k= \sum_{\ell} \gamma_{\ell}\tilde{x}^{i-j}\tilde{e}^{\ell}$, with $\gamma_{\ell}\in {\bf A}$. Similarly one can prove that if $i<j$, we have $\tilde{x}^i\tilde{y}^j\tilde{e}^k= \sum_{\ell} \delta_{\ell}\tilde{y}^{j-i}\tilde{e}^{\ell}$, with $\delta_{{\ell}}\in {\bf A}$. This shows that any element $\tilde{u}$ in $U({\bf A},u,n)$ can be written in the expected form.   

  \noindent Suppose now that:
 $$ \sum_{{\ell>0,k\geq 0}}\alpha_{k,\ell}\tilde{y}^{\ell}\tilde{e}^k+\sum_{m\geq 0,r\geq 0}\beta_{m,r}\tilde{x}^m \tilde{e}^r=0.$$
 Then, as $\tilde{y}^{\ell}\tilde{e}^k\in U({\bf A},u,n)^{-\ell}$ and $\tilde{x}^m \tilde{e}^r\in U({\bf A},u,n)^m$, we deduce from $(5-1-4)$ that 
 \begin{align*}
 \forall  \ell>0 , \, \sum_{k}\alpha_{k,\ell}\tilde{y}^{\ell}\tilde{e}^k=0, \qquad  \forall  m \geq0, \,  \sum_{r}\beta_{m,r}\tilde{x}^m \tilde{e}^r=0. 
 \end{align*}
  Then from  Proposition \ref{prop.U-contient-pol}, we deduce that $\alpha_{k,\ell} =0$ and $\beta_{m,r} =0$.

 \end{proof}
 
 \vskip 10pt
\subsection{Generators and relations for $D(V)^{G'}$}\label{ss-generators-relations}\hfill
\vskip 5pt
Let $ \cal{Z}({\cal T})[t]$ be the polynomials in one variable with coefficients in $ \cal{Z}({\cal T})$.
 From the commutation rules $[E,X]=EX-XE=d_{0}X$ and $[E,Y]=-d_{0}Y$, we easily deduce that for $P\in \cal{Z}({\cal T})[t] $ we have  
 $$YP(E)= P(E+d_{0})Y, \quad  XP(E)=P(E-d_{0})X.\eqno(5-2-1)$$

 From    Proposition \ref{th-key-structure-T0} above, we know that any element in $D(V)^G$ can be written uniquely as a polynomial in $E$ with coefficients in ${\cal Z}({\cal T})$. As $XY$ and $YX$ belong to  $D(V)^G$, there exist therefore two uniquely  determined  polynomials $u_{XY} \text { and } u_{YX}\in {\cal Z}({\cal T})[t]$ such that $XY=u_{XY}(E)$ and $YX=u_{YX}(E)$. From $(5-2-1)$ we obtain that 
 $$YXY=u_{YX}(E)Y= Yu_{XY}(E)=u_{XY}(E+d_{0})Y$$
 and therefore 
 $$u_{YX}(E)=u_{XY}(E+d_{0})\eqno (5-2-2)$$
 
 As the polynomial $u_{XY}$ will play an important role in  Theorem \ref{th-generateurs-relations}  below, let us emphasize the connection between $u_{XY}$ and the Bernstein-Sato polynomial $b_{Y}$. Remark first that $b_{Y}=b_{XY}$.   We know from Corollary \ref{cor-decomp-pol-sym} that 
 \begin{align*}
 &h(XY)(-\lambda+\rho)
 =b_{XY}(\lambda)=b_{Y}(\lambda)\\
 &=\sum_{i=0}^p\alpha_{i}(-\lambda+\rho)(a_{0}d_{0}+a_{1}d_{1}+\dots+a_{r}d_{r})^i=\sum_{i=0}^p\alpha_{i}(-\lambda+\rho)(h(E)(-\lambda+\rho))^{i}
 \end{align*}
 with uniquely defined polynomials $\alpha_{i}\in \C[A]^{W_{0},\tau}$. Therefore we obtain:
  \begin{prop}\label{prop-relation-u-b}\hfill
  
  Keeping the notations above, we have
 $$u_{XY}(t)=\sum_{i=0}^p h^{-1}(\alpha_{i})t^i$$
 \end{prop}
  \begin{theorem}\label{th-generateurs-relations}  
 
Let $f_{XY}(t)= u_{XY}(t+d_{0})-u_{XY}(t)$. The mapping 
 $$\tilde x\longmapsto X,\quad \tilde y \longmapsto Y, \quad \tilde e\longmapsto E$$
 extends uniquely to an isomorphism of ${\cal Z}({\cal T})$-algebras between   $U({\cal Z}({\cal T}), u_{XY},d_{0})$ (which is isomorphic to $S({\cal Z}({\cal T}), f_{XY}, d_{0})/(\Omega_{1})$) and $D(V)^{G'}={\cal T}_{0}[X,Y]$.
 \end{theorem}
 
 \begin{proof}
 
 As $[E,X]=d_{0}X$, $[E,Y]=-d_{0}Y$, $XY=u_{XY}(E)$ and $YX=u_{XY}(E+d_{0})$ (see $(5-2-2)$), and as  from Proposition \ref{th-key-structure-T0} the algebra   $D(V)^{G'}={\cal T}_{0}[X,Y]$ is generated over ${\cal Z}({\cal T})$ by $X,Y,E$, we know (universal property)   that the mapping 
$$\begin{array}{ccc}
 \tilde{x} \longmapsto X,\qquad \tilde{y}  \longmapsto Y,\qquad\tilde{e}   \longmapsto E\\
 \end{array} $$
 extends uniquely to a surjective morphism of ${\cal Z}({\cal T})$-algebras:
 $$\varphi:  U({\cal Z}({\cal T}),u_{XY},d_{0})   \longrightarrow   D(V)^{G'}.$$
 From Corollary \ref{cor.U-basis} any element $\tilde u$ in   $U({\cal Z}({\cal T}),u_{XY},d_{0}) $
 can be written uniquely in the form 
 $$\tilde{u}=\sum_{{\ell>0,k\geq 0}}\alpha_{k,\ell}\tilde{y}^{\ell}\tilde{e}^k+\sum_{m\geq 0,r\geq 0}\beta_{m,r}\tilde{x}^m \tilde{e}^r$$
 with $\alpha_{k,\ell}, \beta_{m,r}\in {\cal Z}({\cal T})$. Suppose now that $\tilde u\in \ker (\varphi)$, then
 $$\varphi(\tilde u)=\sum_{{\ell>0,k\geq 0}}\alpha_{k,\ell}Y^{\ell}E^k+\sum_{m\geq 0,r\geq 0}\beta_{m,r}X^m E^r=0,$$
 with $\alpha_{k,\ell},\beta_{m,r}\in {\cal Z}({\cal T})$. Then Corollary \ref{cor-decomp-D(V)G'} implies that $\alpha_{k,\ell}=\beta_{m,r}=0$. Hence $\varphi$ is an isomorphism.

 \end{proof}
 
  \vskip 20pt
 \section{Radial components}
 \vskip 5pt
\subsection{Radial components and  Bernstein-Sato polynomials}\hfill
 \vskip 5pt

Remember thar for  $\widetilde{\bf a}= (a_{1},a_{2},\dots,a_{r})\in {\bb N}^r$ we have defined  $V_{\widetilde{\bf a}}=V_{(0,a_{1},\dots,a_{r})}$. Remember also that for ${\bf a}=(a_{0},a_{1},\dots,a_{r})$ we have $V_{\bf a}=\Delta^{a_{0}}V_{\widetilde{\bf a}}$. We know   from Proposition \ref{prop-G'-isotypic} that the spaces $U_{\widetilde{\bf a}}=\oplus_{a_{0}\in {\bb N}}\Delta_{0}^{a_{0}}V_{\widetilde{\bf a}}$ are the $G'$-isotypic components of ${\bb C}[V]$ and that the spaces $W_{\widetilde{\bf a}}=\oplus_{a_{0}\in {\bb Z}}\Delta_{0}^{a_{0}}V_{\widetilde{\bf a}}$ are the $G'$-isotypic components of ${\bb C}[{\cal O}]$. Therefore the algebra $D(V)^{G'}={\cal T}_{0}[X,Y]$    stabilizes each space $U_{\widetilde{\bf a}}$ and the algebra $D({\cal O})^{G'}={\cal T}_{0}[X,X^{-1}]={\cal T}$ stabilizes each space $W_{\widetilde{\bf a}}$.

Let us consider the restriction map:

$$\begin{array}{rll} D({\cal O})^{G'}&\longrightarrow& End(W_{\widetilde{\bf a}})\\
D&\longmapsto &r_{_{\widetilde{\bf a}}}(D)=D_{|_{W_{\widetilde{\bf a}}}}
\end{array}
$$
\begin{definition} Let $D\in D({\cal O})^{G'}={\cal T}_{0}[X,X^{-1}]={\cal T}$. The operator $r_{_{\widetilde{\bf a}}}(D)=D_{|_{W_{\widetilde{\bf a}}}}$ is called the radial component of $D$ with respect to $\widetilde{\bf a}$.
\end{definition}

\begin{example} Consider the case where $\widetilde{\bf a}=0$. Then $W_{\widetilde{\bf a}}={\bb C}[\Delta_{0},\Delta_{0}^{-1}]$, and $r_{_{0}}(D) $ is the endomorphism of ${\bb C}[t,t^{-1}]$ defined by $D(\varphi\circ\Delta_{0})= r_{_{0}}(D)(\varphi)\circ \Delta_{0}$. The operator $r_{_{0}}(D)$ is the usual radial component of $D$ (we will see below that it is a differential operator).
\end{example}

Notice now that the space $W_{\widetilde{\bf a}}=\oplus_{a_{0}\in {\bb Z}}\Delta_{0}^{a_{0}}V_{\widetilde{\bf a}}$ can be viewed as the space of Laurent polynomials in $\Delta_{0}$, with coefficients in $V_{\widetilde{\bf a}}$, in other words any $P\in W_{\widetilde{\bf a}}$ can be written uniquely under the form
$$P= \sum \Delta_{0}^p \gamma_{p}$$
with $\gamma_{p}\in V_{\widetilde{\bf a}}$. This can also be written as $P=\varphi\circ(\Delta_{0})$, with $\varphi(t)=\sum t^p\gamma_{p}\in V_{\widetilde{\bf a}}[t, t^{-1}]$ (where $V_{\widetilde{\bf a}}[t, t^{-1}]$ is precisely the set of linear combinations $\sum t^p\gamma_{p}$, with $\gamma_{p}\in V_{\widetilde{\bf a}}$).

There is a natural action of $D({\bb C}^*)={\bb C}[t,t^{-1},t\frac{d}{dt}]$  on $V_{\widetilde{\bf a}}[t, t^{-1}]$ given by $\frac{d}{dt}t^p\gamma_{p}=pt^{p-1}\gamma_{p}.$

\begin{prop}\label{prop-radial=Berntein-Sato}\hfill

Let $D\in  {\cal T}_{n}$ a homogeneous element of degree $n$. Let $b_{D}$ be its Bernstein-Sato polynomial. Let $\varphi\in V_{\widetilde{\bf a}}[t, t^{-1}]$. Then $D(\varphi\circ \Delta_{0})=(t^nb_{D}(t\frac{d}{dt},a_{1},\dots,a_{r})\varphi)\circ\Delta_{0} $, in other words $r_{_{\widetilde{\bf a}}}(D)=t^nb_{D}(t\frac{d}{dt},a_{1},\dots,a_{r})$.
\end{prop}

\begin{proof}

It is enough to show that the two operators coincide on elements of the form $\Delta_{0}^p\gamma_{p}$, with $\gamma_{p}\in V_{\widetilde{\bf a}}$. Then $\varphi=t^p\gamma_{p}$. Let us write
$$b_{D}({\bf a})=\sum_{k}c_{k}(a_{1},\dots,a_{r})a_{0}^k $$ 

We have:
$$\begin{array}{l}
(t^nb_{D}(t\frac{d}{dt},a_{1},\dots,a_{r})\varphi)\circ\Delta_{0}=t^n(\sum_{k}c_{k}(a_{1},\dots,a_{r})(t\frac{t}{dt})^k \varphi)\circ \Delta_{0}\\
=t^n(\sum_{k}c_{k}(a_{1},\dots,a_{r})p^kt^p\gamma_{p})\circ \Delta_{0}\\
=(t^nb_{D}(p,a_{1},\dots,a_{r})t^p\gamma_{p})\circ \Delta_{0}=b_{D}(p,a_{1},\dots,a_{r})\Delta_{0}^{p+n}\gamma_{p}\\
=D(\Delta_{0}^p\gamma_{p})
\end{array}
$$

\end{proof}



\begin{cor} \hfill 

If $(G,V)$ is a $PV$ of commutative parabolic type of rank $r+1$, then the radial component of $Y$ is given by 
$$r_{_{\widetilde{\bf a}}}(Y)= ct^{-1}\prod_{j=0}^r(t\frac{d}{dt}+a_{1}+\dots+a_{j}+j\frac{d}{2})$$
\end{cor}

\begin{proof}

This is just a consequence of the formula for $b_{Y}$ given in Example \ref{example-b_Y-commutative}

\end{proof}

\begin{example}\label{example-det}
Consider the case  $1)$ in example \ref{ex-commutative-PV}. In this case $G=(SL(n)\times SL(n))\times  \C^*$ acting on  $x\in V=M_{n}(\C)$   by $(g_{1},g_{2},t).x=tg_{1}xg_{2}^{-1}$ . Then $\Delta_{0}=X=\det x$ and 
$$Y=\Delta_{0}^*(\partial)=\det(\frac{\partial}{\partial x_{ij}})$$
where $x_{ij}$ are the coefficients of the matrix $X$.
As in this case $\frac{d}{2}=1$ (see \cite{Muller-Rubenthaler-Schiffmann}, table 2 p. 122), we have $b_{Y}(a_{0},a_{1},\dots,a_{n-1})=\prod_{j=0}^{n-1}(a_{0}+a_{1}+\dots+a_{j}+j)$. Therefore the   radial component $ r_{_{0}}(Y)$ defined by $\det(\frac{\partial}{\partial x_{ij}})(\varphi\circ\det)=(r_{_{0}}(Y) \varphi)\circ\det$ is given by 
$$r_{_{0}}(Y)=t^{-1}\prod_{j=0}^{n-1}(t\frac{d}{dt}+ j).$$
This radial component has already been calculated by Ra\"\i s (\cite{Rais}, p.22), by other methods. He obtained that $ r_{_{0}}(Y)=[\prod_{j=2}^{n-1}(t\frac{d}{dt}+ j)]\frac{d}{dt}$. A simple calculation shows that the two operators are the same.

\end{example}
 
 \vskip 15pt
\subsection{Algebras of radial components}
 \vskip 5pt

\begin{definition} The radial component algebra  $R_{\widetilde{\bf a }}$ is the image of $D(V)^{G'}={\cal T}_{0}[X,Y]$ under the map $D\longmapsto r_{_{\widetilde{\bf a}}}(D)$.
\end{definition}
Remember from Proposition \ref{prop-centre-T} that the elements $D$ in $ {\cal Z}({\cal T})$ are characterized by the fact that the corresponding Bernstein-Sato polynomial $b_{D}$ does not depend on the $a_{0}$ variable. Therefore such a $D$ acts by the scalar $b_{D}(0,\widetilde{\bf a})$ on $W_{\widetilde{\bf a}}$, that is $r_{_{\widetilde{\bf a}}}(D)=b_{D}(0,\widetilde{\bf a})Id_{_{W_{\widetilde{\bf a}}}}$.

Let us consider the polynomial $u_{XY}\in {\cal Z}({\cal T})[t]$ which was introduced in section \ref{ss-generators-relations}. If $u_{XY}=\sum_{j}c_{i}t^i$, with $c_{i}\in {\cal Z}({\cal T})$, we define
$$r_{_{\widetilde{\bf a}}}(u_{XY})=\sum_{j}r_{_{\widetilde{\bf a}}}(c_{i})t^i\in {\bb C}[t].$$

\begin{lemma}\label{lemme-injectivite-Y} \hfill

Let ${\bf a}=(a_{0},a_{1}, \dots,a_{r})\in \N^{r+1}$. Suppose that $a_{0}>0$. Then the map $P\longmapsto YP$ from $V_{\bf a}$ to $V_{{\bf a} -1}$ is a $G'$-equivariant isomorphism.

\end{lemma}
\begin{proof} (Sketch) It is enough to prove that this map is not $0$.  As ${\Delta_{0}^*}^{a_{0}}\dots {\Delta_{r}^*}^{a_{r}}$ is the lowest weight vector of $V_{\bf a}^*\subset \C[V^*]$, we have ${\Delta_{0}^*}(\partial)^{a_{0}} \dots {\Delta_{r}^*}(\partial)^{a_{r}}\Delta_{0}^{a_{0}}\dots \Delta_{r}^{a_{r}} (0)\neq0$. Hence ${\Delta_{0}^*}(\partial)\Delta_{0}^{a_{0}}\dots \Delta_{r}^{a_{r}} \neq0$.

\end{proof}

\begin{theorem}\label{th-radial-components}\hfill

The radial component algebra $R_{\widetilde{\bf a}}$ is isomorphic, as an associative algebra over ${\bb C}$,  to the algebra $U({\bb C}, r_{_{\widetilde{\bf a}}}(u_{XY}), d_{0})$ introduced in Definition $\ref{def-Smith-algebra}$.

\end{theorem}

\begin{proof}

The algebra $R_{\widetilde{\bf a}}$ is generated over ${\bb C}$ by the elements $r_{_{\widetilde{\bf a}}}(E),r_{_{\widetilde{\bf a}}}(X), r_{_{\widetilde{\bf a}}}(Y)$. 
The defining relations of $U({\bb C}, r_{_{\widetilde{\bf a}}}(u_{XY}), d_{0})$ are verified:  

 - $[r_{_{\widetilde{\bf a}}}(E),r_{_{\widetilde{\bf a}}}(X)]=r_{_{\widetilde{\bf a}}}([E,X])=d_{0}r_{_{\widetilde{\bf a}}}(X)$

- $[r_{_{\widetilde{\bf a}}}(E),r_{_{\widetilde{\bf a}}}(Y)]=r_{_{\widetilde{\bf a}}}([E,Y])=-d_{0}r_{_{\widetilde{\bf a}}}(Y)$

- $r_{_{\widetilde{\bf a}}}(X)r_{_{\widetilde{\bf a}}}(Y)=r_{_{\widetilde{\bf a}}}(XY)= r_{_{\widetilde{\bf a}}}(u_{XY})(r_{_{\widetilde{\bf a}}}(E))$

- $r_{_{\widetilde{\bf a}}}(Y)r_{_{\widetilde{\bf a}}}(X)=r_{_{\widetilde{\bf a}}}(YX)= r_{_{\widetilde{\bf a}}}(u_{XY})(r_{_{\widetilde{\bf a}}}(E)+d_{0}).$

Therefore the mapping
$$\tilde x\longmapsto r_{_{\widetilde{\bf a}}}(X),\quad \tilde y \longmapsto r_{_{\widetilde{\bf a}}}(Y), \quad \tilde e\longmapsto r_{_{\widetilde{\bf a}}}(E)$$
extends uniquely to a  surjective  morphism of ${\bb C}$-algebras 
$$\varphi_{\widetilde{\bf a}}: U({\bb C}, r_{_{\widetilde{\bf a}}}(u_{XY}),d_{0})\longrightarrow R_{\widetilde{\bf a}}.$$

 From Corollary \ref{cor.U-basis} any element $\tilde u$ in   $U({\bb C}, r_{_{\widetilde{\bf a}}}(u_{XY}),d_{0})$
 can be written uniquely in the form 
 $$\tilde{u}=\sum_{{\ell>0,k\geq 0}}\alpha_{k,\ell}\tilde{y}^{\ell}\tilde{e}^k+\sum_{m\geq 0,s\geq 0}\beta_{m,s}\tilde{x}^m \tilde{e}^s$$
 with $\alpha_{k,\ell}, \beta_{m,s}\in  {\bb C}$. Suppose now that $\tilde u\in \ker (\varphi_{\widetilde{\bf a}})$, then
 
 $$\varphi_{\widetilde{\bf a}}(\tilde u)=\sum_{{\ell>0,k\geq 0}}\alpha_{k,\ell} r_{_{\widetilde{\bf a}}} (Y)^{\ell} r_{_{\widetilde{\bf a}}} (E)^k+\sum_{m\geq 0,s\geq 0}\beta_{m,s} r_{_{\widetilde{\bf a}}} (X)^m  r_{_{\widetilde{\bf a}}} (E)^s=0.$$
 
 Applying this operator to a function of the form $\Delta^{a_{0}}P$, with $P\in V_{\widetilde{\bf a}}$, we obtain:
 $$\sum_{{\ell>0}}Y^{\ell}(\sum_{k\geq 0}\alpha_{k,\ell}   E^k\Delta^{a_{0}}P)+\sum_{m\geq 0}X^m(\sum_{s\geq 0}\beta_{m,s}  E^s\Delta^{a_{0}}P)=0.$$
 
As the operators $X$ and $Y$ have degree $d_{0}$ and $-d_{0}$ respectively, this implies that
$$\forall \ell,\quad Y^\ell(\sum_{k\geq 0}\alpha_{k,\ell}   E^k\Delta^{a_{0}})P=0, \,\,\text{ and }\,\, \forall m,\quad X^m(\sum_{s\geq 0}\beta_{m,s}  E^s\Delta^{a_{0}})P=0.$$

Therefore, by Lemma \ref{lemme-injectivite-Y} we obtain that

$$\forall \ell, a_{0}> \ell\quad \sum_{k\geq 0}\alpha_{k,\ell}   E^k\Delta^{a_{0}}P=0, \,\,\text{ and }\,\, \forall m,\forall a_{0}\quad \sum_{s\geq 0}\beta_{m,s}  E^s\Delta^{a_{0}}P=0.$$

As $E\Delta^{a_{0}}P=(a_{0}d_{0}+d(\widetilde{\bf a}))\Delta^{a_{0}}P$, where $d(\widetilde{\bf a})=a_{1}d_{1}+\dots+a_{r}d_{r}$, we have:
$$\forall \ell, \text{ and } a_{0}>\ell, \,\, \sum_{k\geq 0}\alpha_{k,\ell} (a_{0}d_{0}+d(\widetilde{\bf a}))^k  \Delta^{a_{0}}P=0, \,\,\text{ and } \forall m,\forall a_{0}\,\,  \sum_{s\geq 0}\beta_{m,s}  (a_{0}d_{0}+d(\widetilde{\bf a}))^s\Delta^{a_{0}}P=0.$$

Hence
$$\forall \ell, \text{ and } a_{0}>\ell \,\,\sum_{k\geq 0}\alpha_{k,\ell}   (a_{0}d_{0}+d(\widetilde{\bf a}))^k =0, \,\,\text{ and }\,\, \forall m,\forall a_{0}\,\,\sum_{s\geq 0}\beta_{m,s}  (a_{0}d_{0}+d(\widetilde{\bf a}))^s =0.$$
This implies that $\forall (\ell,k)$ and $\forall (m,s)$, we have $\alpha_{k,\ell}=0$ and $\beta_{m,s}=0.$ Hence $\tilde u=0$ and $\varphi_{\widetilde{\bf a}}$ is injective.

\end{proof}

 \begin{rem} For $\widetilde{\bf a}=0$, the preceding result was first obtained by T. Levasseur (\cite{Levasseur}), by other methods.
 \end{rem}
 
 Define now  $ J_{_{\widetilde{\bf a}}}=\ker  ({r_{_{\widetilde{\bf a}}}}_{|_{D(V)^{G'}}})$.  $ J_{_{\widetilde{\bf a}}}$ is a two-sided ideal of $D(V)^{G'}={\cal T}_{0}[X,Y]$. Remember from Proposition \ref{prop-poly-XY-graduation}
  that any  $D\in D(V)^{G'}$  can be   written uniquely  in the form:
$$ D=\sum_{k\in {\bb N}^* }u_{k}Y^{k}+\sum_{n\in {\bb N}}v_{n}X^{n} \text{   {\rm (finite \,\,sum)}}$$
 
where $u_{k}, v_{n}\in  {\cal   T}_{0}=D(V)^{G}$.
\begin{lemma}\label{lemma-ker-rad}\hfill
$$J_{_{\widetilde{\bf a}}}=\{D=\sum_{k\in {\bb N}^* }u_{k}Y^{k}+\sum_{n\in {\bb N}}v_{n}X^{n}\,|\,u_{k},v_{n}\in J_{_{\widetilde{\bf a}}}\cap {\cal T}_{0}\}.$$
\end{lemma}

\begin{proof}

  From Theorem \ref{th-radial-components} the algebra $R_{\widetilde{\bf a}}$ is isomorphic to $U({\bb C}, r_{_{\widetilde{\bf a}}}(u_{XY}), d_{0})$. If $r_{_{\widetilde{\bf a}}}(D)=\sum_{k\in {\bb N}^* }r_{_{\widetilde{\bf a}}}(u_{k})r_{_{\widetilde{\bf a}}}(Y)^{k}+\sum_{n\in {\bb N}}r_{_{\widetilde{\bf a}}}(v_{n})r_{_{\widetilde{\bf a}}}(X)^{n}=0$, then, from Corollary \ref{cor.U-basis}, we obtain that $r_{_{\widetilde{\bf a}}}(u_{k})=0$ and $r_{_{\widetilde{\bf a}}}(v_{n})=0$ for all $k$ and all $n$.

\end{proof}

  Let us now give a set of generators for the ideal  $\ker  (r_{_{\widetilde{\bf a}}})$ in $D(V)^{G'}={\cal T}_{0}[X,Y]$. 
From  Proposition \ref{prop-radial=Berntein-Sato} we obtain that $\displaystyle r_{_{\widetilde{\bf a}}}(E)= d_{0}(t\frac{d}{dt}) +d(\widetilde{\bf a})$. Therefore  $\displaystyle r_{_{\widetilde{\bf a}}}(\frac{E-d(\widetilde{\bf a})}{d_{0}})=t\frac{d}{dt}$. Define  $\displaystyle G_{i}^{\widetilde{\bf a}}=R_{i}-b_{R_{i}}(\frac{E-d(\widetilde{\bf a}) }{d_{0}},\widetilde{\bf a})$ where the $R_{i}$'s are the Capelli operators introduced in section $2.2$. Obviously $G_{i}^{\widetilde{\bf a}}\in D(V)^G={\cal T}_{0}$. Using Proposition \ref{prop-radial=Berntein-Sato} again we obtain 
$$  \begin{array}{l} \displaystyle r_{_{\widetilde{\bf a}}}(G_{i}^{\widetilde{\bf a}})= r_{_{\widetilde{\bf a}}}(R_{i}-b_{R_{i}}(\frac{E-d(\widetilde{\bf a} }{d_{0}},\widetilde{\bf a}))=r_{_{\widetilde{\bf a}}}(R_{i})- b_{R_{i}}(r_{_{\widetilde{\bf a}}}(\frac{E-d(\widetilde{\bf a})}{d_{0}}), \widetilde{\bf a})\\
 \displaystyle =b_{R_{i}}(t\frac{d}{dt},\widetilde{\bf a})-b_{R_{i}}(t\frac{d}{dt},\widetilde{\bf a})=0.\end{array}$$
 
 Hence the elements $G_{i}^{\widetilde{\bf a}}  $ belong to  $J_{_{\widetilde{\bf a}}}\cap {\cal T}_{0}$.

\begin{theorem}\label{th-ker-rad} \hfill 

 The elements  $G_{i}^{\widetilde{\bf a}}  $ are generators of $J_{_{\widetilde{\bf a}}}$:
$$ J_{_{\widetilde{\bf a}}}= \ker  ({r_{_{\widetilde{\bf a}}}}_{|_{D(V)^{G'}}})=  \sum_{i=0}^rD(V)^{G'}G_{i}^{\widetilde{\bf a}} =
 \sum_{i=0}^r G_{i}^{\widetilde{\bf a}} D(V)^{G'}.
 $$  
\end{theorem}

\begin{proof}


From Lemma \ref{lemma-ker-rad}, it is now enough to prove that 
$$J_{_{\widetilde{\bf a}}}\cap {\cal T}_{0}\subset \sum_{i=0}^rD(V)^{G}G_{i}^{\widetilde{\bf a}} = \sum_{i=0}^r{\cal T}_{0}G_{i}^{\widetilde{\bf a}}  .$$

Let $D\in J_{_{\widetilde{\bf a}}}\cap {\cal T}_{0}$. As ${\cal T}_{0}={\bb C}[R_{0},\dots,R_{r}]$ (Proposition \ref{th-generateurs-op-inv}), we have also ${\cal T}_{0}={\bb C}[G_{0}^{\widetilde{\bf a}},\dots,G_{r}^{\widetilde{\bf a}},E].$

Therefore $D=\sum Q_{i}E^{i}$, where $Q_{i}\in {\bb C}[G_{0}^{\widetilde{\bf a}},\dots,G_{r}^{\widetilde{\bf a}}].$ Hence $ Q_{i}\in Q_{i}(0)+\sum_{i=0}^rD(V)^{G}G_{i}^{\widetilde{\bf a}}$. Then
 $$0=r_{_{\widetilde{\bf a}}}(D)=\sum_{i}Q_{i}(0)r_{_{\widetilde{\bf a}}}(E^{i})=\sum_{i}Q_{i}(0)(d_{0}(t\frac{d}{dt}) +d(\widetilde{\bf a}))^{i}.$$
Therefore $Q_{i}(0)=0$ $(i=0,\dots,r)$. Hence $Q_{i}\in  \sum_{i=0}^rD(V)^{G}G_{i}^{\widetilde{\bf a}}$, which yields $D\in \sum_{i=0}^rD(V)^{G}G_{i}^{\widetilde{\bf a}}$.

\end{proof}

\begin{rem}\label{rem-Levasseur-ker}
For $\widetilde{\bf a}=0$, the result of the preceding Theorem is due to T. Levasseur (\cite{Levasseur}, Theorem 4.11. (v)).
\end{rem}

 \vskip 15pt
\subsection{Rational  radial component algebras}\hfill
 \vskip 5pt

\begin{definition} The rational radial component algebra  $R_{\widetilde{\bf a }}^r$ is the image of $D({\cal O})^{G'}={\cal T}_{0}[X,X^{-1}]={\cal T}$ under the map $D\longmapsto r_{_{\widetilde{\bf a}}}(D)$.
\end{definition}

In fact as shown in the following proposition the structure of the algebras  $R_{\widetilde{\bf a }}^r$ is  more simpler than the structure of  $R_{\widetilde{\bf a }}$, and the ideal    $I_{_{\widetilde{\bf a}}}= \ker  ({r_{_{\widetilde{\bf a}}}})\subset {\cal T}$ has the same generators as  $J_{_{\widetilde{\bf a}}}  $.

\begin{prop}\label{prop-rational-radial-components}\hfill

$1)$ For all $\widetilde{\bf a}$, the rational radial component algebra $R_{\widetilde{\bf a}}^r$ is isomorphic to ${\bb C}[t,t^{-1},t\frac{d}{dt}]$.

$2)$  $I_{_{\widetilde{\bf a}}}=\ker  ({r_{_{\widetilde{\bf a}}}})   =  \sum_{i=0}^r {\cal T}G_{i}^{\widetilde{\bf a}} =
 \sum_{i=0}^r G_{i}^{\widetilde{\bf a}}  {\cal T}.$
\end{prop}

\begin{proof}

$1)$ We have  ${\cal T}={\cal T}_{0}[X,X^{-1}]$.  And ${\cal T}_{0}={\cal Z}({\cal T})[E]$, from Proposition \ref{th-key-structure-T0}. Therefore ${\cal T}={\cal Z}({\cal T})[X,X^{-1},E]$. On the other hand we have ${r_{_{\widetilde{\bf a}}}}({\cal Z}({\cal T}))={\bb C}$, ${r_{_{\widetilde{\bf a}}}}(X)= t$, ${r_{_{\widetilde{\bf a}}}}(X^{-1})= t^{-1}$ and $r_{_{\widetilde{\bf a}}}(E)= d_{0}(t\frac{d}{dt}) +d(\widetilde{\bf a})$. Hence $R_{\widetilde{\bf a}}^r= r_{_{\widetilde{\bf a}}}({\cal T})={\bb C}[t, t^{-1}, d_{0}(t\frac{d}{dt}) +d(\widetilde{\bf a})]={\bb C}[t, t^{-1}, t\frac{d}{dt}]$.

$2)$ Obviously  $ \sum_{i=0}^r {\cal T}G_{i}^{\widetilde{\bf a}}  \subset I_{_{\widetilde{\bf a}}}$. As $I_{_{\widetilde{\bf a}}}$ is a two-sided ideal of ${\cal T}$, it is easily seen to be graded. If $D\in I_{_{\widetilde{\bf a}}}\cap {\cal T}_{p}$, then $X^{-p}D\in {\cal T}_{0}\cap I_{_{\widetilde{\bf a}}}={\cal T}_{0}\cap J_{_{\widetilde{\bf a}}}=\sum_{i=0}^r{\cal T}_{0}G_{i}^{\widetilde{\bf a}} $. Therefore $D\in \sum_{i=0}^r {\cal T}G_{i}^{\widetilde{\bf a}}$.

\end{proof}

\vskip 10pt


\end{document}